\documentclass[11pt]{amsart}
\usepackage{amsmath,amsfonts,amsthm,amssymb,amscd, verbatim, graphicx}

\def\classification#1{\def\@class{#1}}
\classification{\null}

\newcommand{\Fq}{\mathbb{F}_q}
\newcommand{\Fqtwo}{\mathbb{F}_{q^2}}
\newcommand{\Fqthree}{\mathbb{F}_{q^3}}
\DeclareFontFamily{OT1}{rsfs}{}
\DeclareFontShape{OT1}{rsfs}{n}{it}{<-> rsfs10}{}
\DeclareMathAlphabet{\mathscr}{OT1}{rsfs}{n}{it}

\newtheorem{prop}{Proposition}[section]
\newtheorem{thm}[prop]{Theorem}

\newtheorem{cor}[prop]{Corollary}
\newtheorem{lem}[prop]{Lemma}

\numberwithin{equation}{section}

\begin{document}

\title{$(2,m,n)$-groups with Euler characteristic equal to $-2^as^b$}
\author{Nick Gill}
\address{Nick Gill\newline
Department of Mathematics and Statistics\newline
The Open University\newline
Milton Keynes, MK7 6AA\newline
United Kingdom}
\email{n.gill@open.ac.uk}

\begin{abstract}
We study those $(2,m,n)$-groups which are almost simple and for which the absolute value of the Euler characteristic is a product of two prime powers. All such groups which are not isomorphic to $PSL_2(q)$ or $PGL_2(q)$ are completely classified.
\end{abstract}


\maketitle

\section{Introduction}

Let $m$ and $n$ be positive integers. A {\it $(2,m,n)$-group} is a triple $(G,g,h)$ where $G$ is a group, $g$ (resp. $h$) is an element of $G$ of order $m$ (resp. $n$), and $G$ has a presentation of form
\begin{equation}\label{e: meteor}
\langle g, h \, \mid \, g^m = h^n = (gh)^2 = \cdots = 1\rangle.
\end{equation}
We will often abuse notation and simply refer to the group $G$ as a $(2,m,n)$-group. 

The Euler characteristic $\chi$ of a $(2,m,n)$-group $G$ is defined by the formula
\begin{equation}\label{e: sunny}
\chi = |G|\left(\frac1{m}-\frac12+\frac1{n}\right) = -|G|\frac{mn-2m-2n}{2mn}.
\end{equation}
It is well known that $\chi$ is an even integer and, moreover, that $\chi\leq 2$.

In this paper we investigate the situation where $\chi=-2^as^b$ for some odd prime $s$ and positive integers $a$ and $b$. We are interested in understanding the structure of the finite $(2,m,n)$-group $(G,g,h)$ in such a situation, particularly when $G$ is non-solvable. 

This paper is a follow-up to an earlier paper \cite{gilla}. In particular the main result of this paper was first announced there, and many of the basic ideas needed to prove our main result are first introduced there. We will, therefore, keep exposition and motivation to a minimum in this paper, and refer the reader to \cite{gilla} for futher details.

\subsection{Main result}

To state our main result we need a definition: a group $S$ is {\it almost simple} if it contains a finite, simple, non-abelian, normal subgroup $T$ such that $T\leq S \leq Aut T$. We call $T$ the {\it socle} of $S$. A $(2,m,n)$-group $(S,g,h)$ is {\it almost simple} if $S$ is almost simple. Now we can state the main result of this paper.

\begin{thm}\label{t: almost simple two primes}
Let $(S,g,h)$ be an almost simple $(2,m,n)$-group with Euler characteristic $\chi=-2^as^b$ for some odd prime $s$ and integers $a,b\geq 1$. Let $T$ be the unique non-trivial normal subgroup in $S$. Then one of the following holds:
\begin{enumerate}
 \item $T=PSL_2(q)$ for some prime power $q\geq 5$ and either $S=T$ or $S=PGL_2(q)$ or else one of the possibilities listed in Table \ref{table: main1} holds.
\begin{center}
\begin{table}
 \begin{tabular}{|c|c|c|}
 \hline
Group & $\{m,n\}$ & $\chi$ \\
\hline
$PSL_2(9).2\cong S_6$ & $\{5,6\} $ & $-2^5\cdot 3$ \\
$PSL_2(9).(C_2\times C_2)$ & $\{4,10\} $ & $-2^3 \cdot 3^3$ \\
$PSL_2(25).2$ & $\{6,13\}$ & $-2^5\cdot 5^3$ \\
\hline
 \end{tabular}
\caption{Some $(2,m,n)$-groups for which $\chi=-2^as^b$}\label{table: main1}
\end{table}
\end{center}

\item $S=T$, $T.2$ or $T.3$, where $T$ is a finite simple group and all possibilities are listed in Table \ref{table: main}.
\begin{center}
\begin{table}
 \begin{tabular}{|c|c|c|}
 \hline
Group & $\{m,n\}$ & $\chi$ \\
\hline
$T=SL_3(3)$& $\{4,13\}$& $-|S:T|\cdot2^2\cdot 3^5$ \\
  $S=SL_3(3)$& $\{13,13\}$& $-2^3\cdot 3^5$ \\
  $S=SL_3(5)$& $\{3,31\}$& $-2^4\cdot 5^5$ \\
  $S=PSL_3(4).2$& $\{5,14\}$& $ -2^{10}\cdot 3^2$ \\
  $S=PSL_3(4).2$& $\{10,7\}$& $ -2^7\cdot 3^4$ \\
  $S=PSL_3(4).3$& $\{15,21\}$& $ -2^5\cdot 3^6$ \\
  $S=SU_3(3).2$& $\{4,7\}$& $-2^4\cdot 3^4$ \\
  $T=SU_3(3)$& $\{6,7\}$& $-|S:T|\cdot2^7\cdot 3^2$ \\
  $S=SU_3(3)$& $\{7,7\}$& $-2^4\cdot 3^4$ \\
  $S=SU_3(4).2$& $\{6,13\}$& $-2^8\cdot 5^3$ \\
 $S=PSU_3(8)$& $\{7,19\}$ & $-2^8\cdot 3^8$ \\
$S=G_2(3).2$& $\{13,14\}$& $-2^{12}\cdot3^6$ \\
 $S=Sp_6(2)$& $\{7,10\}$& $-2^9\cdot 3^6$ \\
 $S=PSU_4(3).2$& $\{5,14\}$& $-2^{11}\cdot 3^6$ \\
 $S=PSU_4(3).2$& $\{10,7\}$& $-2^8\cdot 3^8$ \\
 $S=SL_4(2).2=S_8$& $\{10,7\}$& $-2^7\cdot3^4$ \\
 $S=S_7$& $\{10,7\}$& $-2^4\cdot3^4$ \\
 $S=A_9$& $\{10,7\}$& $-2^6\cdot3^6$ \\
$T=SU_4(2)$ & $\{5,6\}$ & $ -|S:T|\cdot2^7\cdot3^3$\\
$S=SU(4,2).2$ & $\{10,4\}$ & $-2^5\cdot 3^5$ \\
$S=SU(4,2).2$ & $\{10,5\}$ & $-2^7\cdot 3^4$ \\
$S=SU(4,2).2$ & $\{10,10\}$ & $-2^6\cdot 3^5$ \\
\hline 
 \end{tabular}
\caption{Some $(2,m,n)$-groups for which $\chi=-2^as^b$}\label{table: main}
\end{table}
\end{center}
\end{enumerate}
What is more a $(2,m,n)$-group exists in each case listed in Tables \ref{table: main1} and \ref{table: main}.
\end{thm}

Some comments about Tables \ref{table: main1} and \ref{table: main} are in order. Note, first, that for those entries of Table \ref{table: main} where we specify only $T$ (rather than $S$), there are two $(2,m,n)$-groups $(S,g,h)$ in each case: one where $S=T$ and one where $S=T.2$. 

Secondly, we note that the single degree $3$ extension and the single degree $4$ extension listed in the two tables are uniquely defined: up to isomorphism there is only one almost simple group $PSL_3(4).3$ and one almost simple group $PSL_2(9).(C_2\times C_2)$; the same comment is also true for many of the degree $2$ extensions listed, but not all. However, consulting \cite{atlas} we find that, in all but one case, the requirement that $S=T.2$ is generated by two elements of orders $m$ and $n$ prescribes the group uniquely, up to isomorphism. (In particular we observe that the entry with group $PSL_2(25).2$ in Table \ref{table: main1} is distinct from $PGL_2(25)$.) 

The non-unique case is as follows: there are three distinct groups $S=PSU_4(3).2$, all of which occur as $(2,7,10)$-groups (these are all of the almost simple degree 2 extensions of $PSU_4(3)$). 

Theorem \ref{t: almost simple two primes} is proved in \S\ref{s: almost simple}. To prove Theorem \ref{t: almost simple two primes} we make use of a more general (but weaker) result which is given in \S\ref{s: double prime}. This result gives a structure statement for finite $(2,m,n)$-groups $G$ with Euler characteristic $\chi$ equal to $-2^as^b$ for some odd prime $s$ (note that, since $\chi$ is always even, this is the only possibility when $\chi$ is divisible by exactly two distinct primes).

Theorem \ref{t: almost simple two primes} comes close to classifying all almost simple $(2,m,n)$-groups with $\chi$ as given, however the case when $S=PSL_2(q)$ or $PGL_2(q)$ is not fully enumerated. On the other hand the general question of when $PSL_2(q)$ or $PGL_2(q)$ are $(2,m,n)$-groups has been studied in \cite{sah} and a complete answer to this question can be found there. Ascertaining when these groups have Euler characteristic divisible by exactly two distinct primes reduces to some difficult number-theoretic questions; we discuss these, along with other open questions, in \S\ref{s: final}.

\subsection{Acknowledgments}

I wish to thank Robert Brignall, John Britnell, Ian Short and Jozef {{\v{S}}ir{\'a}{\v{n}} for many useful discussions.

I owe a particular debt to Marston Conder who generously shared his mathematical and computational insight; in particular many of the computer calculations referred to in \S\ref{s: existence} are due to him.

Finally I have been a frequent visitor to the University of Bristol during the period of research for this paper, and I wish to acknowledge the generous support of the Bristol mathematics department.

\section{Background on groups}\label{s: background}

In this section we add to the notation already esablished, and we present a number of well-known results from group theory that will be useful in the sequel.

The following notation will hold for the rest of the paper: $(G,g,h)$ is always a finite $(2,m,n)$-group; $(S,g,h)$ is always a finite almost simple $(2,m,n)$-group; $T$ is always a simple group. We use $\chi$ or $\chi_G$ to denote the Euler characteristic of the group $G$. 

For groups $H, K$ we write $H.K$ to denote an extension of $H$ by $K$; i.e. $H.K$ is a group with normal subgroup $H$ such that $H.K/ H\cong K$. In the particular situation where the extension is split we write $H\rtimes K$, i.e. we have a semi-direct product. For an integer $k$ write $H^k$ to mean $\underbrace{H\times \cdots \times H}_k$. 

For an integer $n>1$ we write $C_n$ for the cyclic group of order $n$ and $D_n$ for the dihedral group of order $n$. We also sometimes write $n$ when we meet $C_n$ particularly when we are writing extensions of simple groups; so, for instance, $T.2$ is an extension of the simple group $T$ by a cyclic group of order $2$. 

Let $K$ be a group and let us consider some important normal subgroups. For primes $p_1, \dots, p_k$, write $O_{p_1, \dots, p_k}(K)$ for the largest normal subgroup of $K$ with order equal to $p_1^{a_1}\cdots p_k^{a_k}$ for some non-negative integers $a_1, \dots, a_k$; in particular $O_2(K)$ is the largest normal $2$-group in $K$. We write $Z(K)$ for the centre of $K$.


Let $a$ and $b$ be positive integers. Write $(a,b)$ for the greatest common divisor of $a$ and $b$, and $[a,b]$ for the lowest common multiple of $a$ and $b$; observe that $ab=[a,b](a,b)$. For a prime $p$ write $a_p$ for the largest power of $p$ that divides $a$; write $a_{p'}$ for $a/a_p$. For fixed positive integers $q$ and $a$ we define a prime $t$ to be a {\it primitive prime divisor for $q^a-1$} if $t$ divides $q^a-1$ but $t$ does not divides $q^i-1$ for any $i=1, \dots, a-1$. For fixed $q$ we will write $r_a$ to mean a primitive prime divisor for $q^a-1$; then we can state (a version of) Zsigmondy's theorem \cite{zsig}:

\begin{thm}\label{t: zsig}
Let $q$ be a positive integer. For all $a>1$ there exists a primitive prime divisor $r_a$ unless
\begin{enumerate}
 \item $(a,q)=(6,2)$;
\item $a=2$ and $q=2^b-1$ for some positive integer $b$.
\end{enumerate}
\end{thm}

Note that $r_1$ exists whenever $q>2$; note too that, although $r_2$ does not always exist, still, for $q>3$, there are always at least two primes dividing $q^2-1$. The following result is of similar ilk to Theorem \ref{t: zsig}; it is Mih{\u{a}}ilescu's theorem \cite{mihailescu} proving the Catalan conjecture.

\begin{thm}\label{t: catalan}
Suppose that $q=p^a$ for some prime $p$ and positive integer $a$. If $q=2^a\pm1$ and $q\neq p$, then $q=9$.
\end{thm}

If $T$ is a finite simple group of Lie type, then we use notation consistent with \cite[Definition 2.2.8]{gls3} or, equivalently, with \cite[Table 5.1.A]{kl}. In particular we write $T=T_n(q)$ to mean that $T$ is of rank $n$, and $q$ is a power of a prime $p$ (in particular, for the Suzuki-Ree groups, we choose notation so that $q$ is an integer). Using this definition, certain groups are excluded because they are non-simple, namely
$$A_1(2), A_1(3), {^2A_2(2)}, {^2B_2(2)}, C_2(2), {^2F_4(2)}, G_2(2), {^2G_2(3)},$$
and $C_2(2), {^2F_4(2)}, G_2(2), {^2G_2(3)}$ are replaced by their derived subgroups.

We need two results that follow from the Lang-Steinberg theorem. For the first we suppose that $T=T_n(q)$ is an untwisted group of Lie type; then $T$ is the fixed set of a Frobenius endomorphism of a simple algebraic group $T_n(\overline{\Fq})$. Suppose that $\zeta$ is a non-trivial field automorphism of $T$ or, more generally, the product of a non-trivial field automorphism of $T$ with a graph automorphism of $T$. Observe that $\zeta$ can be thought of as a restriction of an endomorphism of $T_n(\overline{\Fq})$; what is more this endomorphism has the particular property that it has a finite number of fixed points. With the notation just established the Lang-Steinberg theorem implies the following result:

\begin{prop}\label{p: lang steinberg}
Any conjugacy class of $T$ which is {\it stable} under $\zeta$ (i.e. is stabilized set-wise) must intersect $X$ non-trivially, where $X$ is the centralizer in $T_n(q)$ of $\zeta$.
\end{prop}
\begin{proof}
This is well known; see for instance \cite[3.10 and 3.12]{dignemichel}. 
 \end{proof}

The Lang-Steinberg theorem also applies to twisted groups, however we will only need it when $T={^2B_2(q)}$, a twisted group of Lie type and $\delta$ is a field automorphism. The situation here is very similar: we observe first that $\delta$ can be thought of as a restriction of an endormorphism of the connected algebraic group $B_2(\overline{F_q})$ (restricted first to act on $B_2(q)\cong P\Omega_5(q)$ and then restricted again to act on $T$) and, again, this endomorphism has a finite number of fixed points. Now the Lang-Steinberg theorem implies the following:

\begin{prop}\label{p: lang steinberg 2b2}
Any conjugacy class of $B_2(q)$ which is stable under $\delta$ must intersect the subfield subgroup $B_2(q_0)$ non-trivially, where $B_2(q_0)$ is the centralizer in $B_2(q)$ of $\delta$.
\end{prop}

For $g$ an element of a group $K$ write $o(g)$ for the order of $g$; write $g^K$ to mean the conjugacy class of $g$ in $K$; write ${\mathrm Irr}(K)$ for the set of irreducible characters of $K$. The following proposition appears as an exercise in \cite[p. 45]{isaacs}.

\begin{prop}\label{p: character}
 Let $g,h,z$ be elements of a group $K$. Define the integer
$$a_{g,h,z}=\left|\{(x,y)\in g^K\times h^K \mid xy=z\}\right|.$$
 Then
$$a_{g,h,z} = \frac{|K|}{|C_K(g)|\cdot |C_K(h)|} \sum_{\chi \in {\mathrm Irr}(K)} \frac{\chi(g)\chi(h)\overline{\chi(z)}}{\chi(1)}.$$
\end{prop}

\section{Lemmas on primes}\label{s: bound}

Recall that $(G,g,h)$ is a $(2,m,n)$-group with Euler characteristic $\chi$. In this section we recall some results from \cite{gilla} concerning primes dividing $\chi$. 

\begin{lem}\cite[Lemma 3.1]{gilla}\label{l: lcm}
Suppose that $t$ is an odd prime dividing $|G|$. If $|G|_t>|[m,n]|_t$, then $t$ divides $\chi_G$.
\end{lem}

An immediate corollary is the following result which is a particular case of Lemma 3.2 in \cite{cps}. 

\begin{lem}\label{l: sylow}
Suppose that $t$ is an odd prime divisor of $|G|$ such that a Sylow $t$-subgroup of $G$ is not cyclic. Then $|G|_t>|[m,n]|_t$ and, in particular, $t$ divides $\chi_G$.
\end{lem}

Let $N$ be a normal subgroup of $G$. Define $m_N$ (resp. $n_N)$ to be the order of $gN$ (resp. $hN$) in $G/N$.

\begin{lem}\cite[Lemma 3.3]{gilla}\label{l: normal}
 Let $N$ be a normal subgroup of the $(2,m,n)$-group $(G,g,h)$. If $G/N$ is not cyclic then $(G/N, gN, hN)$ is a $(2, m_N,n_N)$-group.
\end{lem}

\begin{lem}\cite[Lemma 3.4]{gilla}\label{l: divisible}
Let $N$ be a normal subgroup of $G$. If an odd prime $t$ satisfies $|G/N|_t>|[m_N, n_N]|_t$ then $t$ divides $\chi_G$.
\end{lem}

Finally we state the main result which we will use in \S\ref{s: double prime}. First some notation. Given a finite group $K$, let $\pi(K)$ be the set of all prime divisors of its order; let $\pi_{nc}(K)\subseteq \pi(K)$ to be the set of primes for which the corresponding Sylow-subgroups of $K$ are non-cyclic; let $\pi_c(K)= \pi(K)\backslash\pi_{nc}(K)$. A subset $X\subset \pi(K)$ is called an {\it independence set} if, for all distinct $p,q\in X$ there exists no element in $K$ of order $pq$. Write $t(K)$ (resp. $t_c(K)$) for the maximum size of an independence set in $\pi(K)$ (resp. $\pi_c(K)$).

\begin{prop}\cite[Proposition 3.6]{gilla}\label{p: gk 2}
Let $(G,g,h)$ be a finite $(2,m,n)$-group of even order. Let $N$ be a normal subgroup of $G$ with non-cyclic Sylow $2$-subgroups. Then the number of primes dividing $\frac{|G|}{[m,n]}$ is at least
$$\max\{0, t_c(N)-2\} + |\pi_{nc}(N)|.$$ The number of primes dividing $\frac{|G|}{[m,n]}$ is also at least $t(N)-2$.
\end{prop}

Note that Lemma~\ref{l: lcm} implies that if a prime divides $\frac{|G|}{[m,n]}$ then it divides $\chi_G$.

\section{Groups with $\chi=-2^as^b$}\label{s: double prime}

In this section we consider those $(2,m,n)$-groups $(G,g,h)$ such that $\chi_G$ is divisible by exactly two distinct primes. The first two results reduce the question to studying almost simple groups satisfying a particular property.

\begin{prop}\label{p: two primes}
Let $G$ be a non-solvable finite $(2,m,n)$-group with Euler characteristic $\chi$ divisible by exactly two primes, $2$ and $s$. Write $\overline{G}=G/O_{2,s}(G)$. 

Then $\overline{G}$ has a normal subgroup isomorphic to $M\times T_1\times\cdots T_k$ where $M$ is solvable with a cyclic Fitting subgroup of odd order, $k$ is a positive integer, $T_1,\dots, T_k$ are simple groups such that, for all $i\neq j$, $(|T_i|, T_j|)=2^as^b$ for some non-negative integers $a,b$, and $\overline{G}/(M\times T_1\dots T_k)$ is isomorphic to a subgroup of ${\mathrm Out}(T_1\times\dots\times T_k)$.
\end{prop}
\begin{proof}
The proof is entirely analogous to that of \cite[Proposition 4.1]{gilla} using in addition the fact that $O_{2,s}(G)$ is solvable (a consequence of Burnside's $p^aq^b$-theorem).
\end{proof}

\begin{cor}
For $i=1,\dots, k$, there exists an almost simple group $S_i$ with socle $T_i$ such that $S_i$ is a $(2,m_i, n_i)$-group and $\frac{|S_i|}{[m_i, n_i]} = \pm 2^as^b$ for some non-negative integers $a$ and $b$.
\end{cor}
\begin{proof}
Note that, since $(|T_i|, T_j|)=2^as^b$ for distinct $i$ and $j$, we have $T_i\not\cong T_j$ for distinct $i$ and $j$. Thus $T_i\unlhd \overline{G}$ for all $i=1,\dots, k$ and so $C_{\overline{G}}(T_i)\unlhd\overline{G}$ for all $i=1,\dots, k$. Moreover, for $i=1,\dots, k$, $\overline{G}/C_{\overline{G}}(T_i)$ is isomorphic to $S_i$, an almost simple group with socle $T_i$.

Now Lemma \ref{l: divisible} implies that, for some integers $m_i$ and $n_i$, the group $S_i$ is an almost simple $(2,m_i,n_i)$-subgroup such that $\frac{|S_i|}{[m_i, n_i]} = \pm 2^as^b$ for some non-negative integers $a$ and $b$.
\end{proof}

Our remaining task is to study those almost simple groups $(2,m,n)$-groups $S$ such that $\frac{|S|}{[m,n]}=\pm 2^as^b$ for some non-negative integers $a$ and $b$. The next two results give all possibilities.

\begin{lem}\label{l: lie 2 prime}
Let $S$ be a finite almost simple group with socle $T=T_n(q)$ is simple of Lie type of rank $n$. Suppose $S$ is a $(2,m,n)$-group such that $\frac{|S|}{[m,n]}=2^as^b$  for some non-negative integers $a$ and $b$. We list the possible isomorphism classes for $T$, along with restrictions on $s$.

\begin{center}
\begin{tabular}{|c|c|}
 \hline
$T$ & Restrictions on $s$ \\
\hline
$A_n(q) \cong PSL_{n+1}(q), n=1,2$ & \\
${^2A_2}(q)\cong PSU_3(q)$ & \\
${^2B_2(2^{2x+1})}, x\in\mathbb{Z}^+$ & $s\neq 3$ \\
$A_3(2)\cong SL_4(2), A_3(3) \cong PSL_4(3)$ & $s=3$ \\
${^2A_3}(2) \cong SU_4(2), {^2A_3}(3)\cong PSU_4(3), {^2A_4}(2)\cong SU_5(2)$ & $s=3$ \\
$C_3(2)\cong Sp_6(2)$ & $s=3$ \\
$G_2(3)$ & $s=3$\\
\hline
\end{tabular}
 \end{center}
\end{lem}
\begin{proof}
Proposition \ref{p: gk 2} implies that 
\begin{equation}\label{e: coolcool}
\max\{0, t_c(T)-2\} + |\pi_{nc}(T)|\leq 2. 
\end{equation}
Now \cite[Proposition 3.7]{gilla} gives a list of simple groups of Lie type satisfying \eqref{e: coolcool}; these are the possibilities that we must consider. In addition to the groups listed above we must rule out
\begin{equation}\label{baddies}
A_4(2), B_3(3), C_2(q)', C_3(3), C_4(2), D_4(2), {^2D_4(2)}, F_4(2).
\end{equation}

Assume, then, that $T$ is congruent to one of the groups listed in \eqref{baddies}. In what follows we make frequent use of \cite[Proposition 2.9.1 and Theorem 5.1.1]{kl} in which all isomorphisms between low rank groups of Lie type are listed.

We attend to the infinite family in (\ref{baddies}) first (note that we write $C_2(q)'$ for the derived subgroup of $C_2(q)$ to take into account the fact that $C_2(q)\cong Sp_4(2)$ is not simple). If $T=C_2(q)'$, then $T$ has non-cyclic Sylow $t$-subgroups for $t=p, t_1, t_2$ (where $t_1, t_2$ are distinct primes dividing $q^2-1$); thus Lemma~\ref{l: sylow} implies that we can rule out this situation whenever $t_1$ and $t_2$ exist, i.e. whenever $q>3$. If $T=C_2(3)$ then $T\cong {^2A_3}(3)$ and is already listed; if $T=C_2(2)$ then $T\cong A_1(9)\cong PSL_2(9)$ which is already listed.

To rule out the remaining groups in (\ref{baddies}) we present the following table. For each group $T$ we list a set of primes which lie in $\pi_{nc}(T)$ and an independence set in $\pi(T)$; our sources are \cite{atlas, vasvdov}. In every case we obtain a contradiction with Proposition \ref{p: gk 2}.

\begin{center}
\begin{tabular}{|c|c|c|}
\hline
$T$ & Primes in $\pi_{nc}(T)$ & An independence set in $\pi(T)$ \\
\hline
$A_4(2)$ & 2,3 & 5,7,31 \\
$B_3(3)$ & 2,3 & 5,7,13 \\
$C_3(3)$ & 2,3 & 5,7,13 \\
$C_4(2)$ & 2,3,5 & \\
$D_4(2)$ & 2,3,5 & \\
${^2D_4}(2)$ & 2,3 & 5,7,17 \\
$F_4(2)$ & 2 & 5,7,13,17 \\
\hline
\end{tabular}
\end{center}

We must now prove the listed restrictions on $s$. That $s\neq 3$ for $T={^2B_2(2^{2x+1})}$ follows from the fact that $s$ does not divide $|T|$ and that ${^2B_2(2^{2x+1})}$ is not listed in \cite[Proposition 4.3]{gilla}. That $s=3$ for the last four lines follows from the fact that Sylow $3$-subgroups are non-cyclic in every case.
\end{proof}

\begin{lem}\label{l: sporadic 2 prime}
Let $S$ be a finite almost simple group with socle $T$ where $T$ is not a finite simple group of Lie type. Suppose $S$ is a $(2,m,n)$-group such that $\frac{|S|}{[m,n]}=2^as^b$ for some non-negative integers $a$ and $b$. Then $T$ is isomorphic to one of the following:
\begin{enumerate}
 \item the alternating groups $A_n$ for $n=7,9$ (and $s=3$);
\item the sporadic groups $M_{11}$ and $M_{12}$ (and $s=3$).
\end{enumerate}
\end{lem}
\begin{proof}
If $n\geq 10$ then the Sylow $t$-subgroups of $A_n$ are non-cyclic for $t=2,3$ and $5$. Then Lemma~\ref{l: sylow} yields a contradiction. We conclude that, if $T\cong A_n$ is alternating, then $n\leq 9$. Now, by \cite[Proposition 2.9.1]{kl}, $A_5, A_6$ and $A_8$ are all isomorphic to finite simple groups of Lie type; this leaves $n=7$ and $9$.

We examine \cite[Table 2]{vasvdov} and rule out all sporadic simple groups $T$ with $t(T)>4$ or else $t(T)=4$ and there is an independent set of size 4 with all primes odd. This leaves $M_{11}, M_{12}, J_2, J_3, He, McL, HN$ and $HiS$. Of these, all but $M_{11}, M_{12}$ and $J_3$ have non-cyclic Sylow $t$-subgroups for $t=2,3,5$. Furthermore $J_3$ has non-cyclic Sylow $t$-subgroups for $t=2$ and $3$, and also has an independence set $\{5,17,19\}$. This leaves $M_{11}$ and $M_{12}$ as listed.
\end{proof}

The results of this section give necessary conditions for a group $G$ to be a $(2,m,n)$-group such that $\chi_G$ is divisible by exactly two distinct primes. We can strengthen these results with some simple observations.

First of all, under the assumptions of Lemma \ref{l: lie 2 prime}, if $T=PSL_3(q)$ for some odd prime $q$, then Lemma \ref{l: lcm} implies that $q=2^a-1$ for some integer $a\geq 2$; similarly $q=2^a+1$ when $T=PSU_3(q)$ with $q$ odd. Using Mih{\u{a}}ilescu's theorem (Theorem \ref{t: catalan}) one can also give conditions on even $q$ for both $T=PSL_3(q)$ and $T=PSU_3(q)$.

Secondly, we observe that all of the groups listed in Lemmas \ref{l: lie 2 prime} and \ref{l: sporadic 2 prime} have order divisible by 2,3 and 5, except for $G_2(3)$ and, possibly, $A_1(q)$, $A_2(q), {^2A_2(q)}$, and ${^2B_2(q)}$. This gives strong conditions on the groups $T_1, \dots, T_k$ in Proposition \ref{p: two primes} since $(|T_i|, T_j|)= 2^as^b$ for all $i\neq j$.

\section{$S$ is almost simple and $\chi$ is a product of two primes}\label{s: almost simple}

In this section we assume the following: $S$ is a finite almost simple group with socle $T$; furthermore $(S,g,h)$ is a $(2,m,n)$-group with corresponding Euler characteristic $\chi$ such that $\chi=2^as^b$ for some prime $s$ and positive integers $a$ and $b$. We write $\Lambda=\{m,n\}$.

Now Lemma~\ref{l: lcm} implies that $\frac{|S|}{[m,n]}$ is divisible by at most two primes. Thus the possible isomorphism classes of $T$ are listed in Lemmas \ref{l: lie 2 prime} and \ref{l: sporadic 2 prime}. 

Our job in this sectios is to prove Theorem \ref{t: almost simple two primes}. The proof proceeds in the following way: in Sections \ref{s: 61} to \ref{s: sporadics} we go through the different possible isomorphism classes for an almost simple group $S$ that are compatible with Lemmas \ref{l: lie 2 prime} and \ref{l: sporadic 2 prime}; in each case we produce a finite list of triples $(S,m,n)$ such that $S$ may possibly occur as a $(2,m,n)$-group $(S,g,h)$ for some $g,h\in S$. Then, in \S\ref{s: existence}, we go through each of the listed possibilities and establish which really occur, i.e. when there are elements $g,h\in S$ such that $(S,g,h)$ is a $(2,m,n)$-group. 

\subsection{$T=PSL_2(q)$}\label{s: 61}

In this situation we need only deal with the situation when $S\neq PSL_2(q)$ or $PGL_2(q)$. We introduce some notation: $q=p^a$ for some prime $p$; $T'$ is a group isomorphic to either $PSL_2(q)$ or $PGL_2(q)$ (these coincide with $T$ when $q$ is even); $\delta$ is a field automorphism of $T$ (hence also of $T'$) of order $a_1>1$; we write $S'=\langle T', \delta \rangle=T'\rtimes \langle \delta \rangle$. We choose $T'$ and $\delta$ so that $S$ is a subgroup of $S'$ of index at most $2$; if $S$ contains $PGL_2(q)$ or $\delta$ then we can choose $T'$ and $\delta$ so that $S=S'$, otherwise $S=\langle T, \delta\epsilon\rangle$ where $\epsilon$ is a diagonal automorphism of $T$, and $|S:T|=a_1$. 

Observe that $\chi=-2^ap^b$, since the Sylow $p$-subgroups of $T'$ are non-cyclic whenever $T'$ admits a non-trivial field automorphism. Let $x$ be the order of an element of $T'$; then $x$ divides at least one of $q-1, p, q+1$. Let $u$ be an element of $S'$ such that $uT'$ generates $S'/T'$; in semi-direct product notation $u=(t, \delta)$ for some $t\in T'$. The Lang-Steinberg theorem (Proposition \ref{p: lang steinberg}) implies that $u$ has order dividing one of $a_1(q_1-1), a_1p, a_1(q_1+1)$ where $q=q_1^{a_1}$.

Suppose that $\Lambda=\{\lambda_1, \lambda_2\}$; then we may assume that $\lambda_1$ divides one of the orders $a_1(q_1-1), a_1p, a_1(q_1+1)$ where $a_1$ divides $a$ and $q_1$ is such that $q=q_1^{a_1}$; similarly $\lambda_2$ divides either $a_2(q_2-1), a_2p$ or $a_2(q_2+1)$ for some $a_2|a$ and $q_2$ is such that $q_2^{a_2}=q$. We assume, without loss of generality, that $a_1\geq a_2$.

In what follows, for an integer $k$ we write $r_{p,k}$ for a primitive prime divisor of $p^k-1$ (i.e. it is primitive with respect to the prime $p$ rather than with respect to $q$, as we have written elsewhere).

\begin{lem}
 $a_1=2$.
\end{lem}
\begin{proof}
Suppose that $a_1>2$. Then the condition $(gh)^2=1$ (which implies that $(ghT')^2=1\in S'/T'$) implies that $a_2>1$. Consider the primes $r_{p,a}$ and $r_{p,2a}$; since $a\geq a_1>2$, Theorem \ref{t: zsig} implies that at least one of these exist. Observe furthermore that neither $r_{p,a}$ nor $r_{p, 2a}$ divide $(q_1-1)(q_1+1)(q_2-1)(q_2+1)$. Furthermore, by Fermat's little theorem, $a_i|p^{a_i-1}-1$ for $i=1,2$ and we conclude that neither $r_{p,a}$ nor $r_{p, 2a}$ divide
$$a_1a_2(q_1-1)(q_1+1)(q_2-1)(q_2+1).$$
But Lemma \ref{l: lcm} implies that $r_{p,a}$ and $r_{p, 2a}$ must divide $\lambda_1\lambda_2$ and we have a contradiction.
\end{proof}

\begin{lem}
 If $p$ is odd then one of the following holds:
\begin{enumerate}
 \item $S=PSL_2(25).2$, $\Lambda=\{6,13\}$, $\chi=-2^5\cdot 5^3$;
\item $S=PSL_2(9).2$, $\Lambda=\{4,5\} $, $\chi = -2^2\cdot 3^2$;
\item $S=PSL_2(9).2$, $\Lambda=\{5,6\} $, $\chi = -2^4\cdot 3^2 $;
\item $S=PSL_2(9).(C_2\times C_2)$, $\Lambda=\{4,10\} $, $\chi = -2^3 \cdot 3^3$;
\end{enumerate}
In all cases the group $S$ is distinct from $PSL_2(q)$ and $PGL_2(q)$.
\end{lem}
\begin{proof}
Assume first that $q>25$, and we show a contradiction. Since $a_1=2$ we know that $\lambda_1$ divides one of $2(\sqrt{q}+1)$, $2(\sqrt{q}-1)$, $2p$. If $a_2=2$ then the same can be said of $\lambda_2$. But now write $a=2b$ and observe that $r_{p, 4b}$ divides $q+1$; then if $a_2=2$ this implies that $r_{p,4b}$ does not divide $\lambda_1\lambda_2$ which is a contradiction. We conclude that $a_2=1$ and, moreover, $\lambda_2$ divides $q+1$ and is divisible by $r_{p, 4b}$.

Now if $\lambda_1$ divides $2p$ then we conclude that $|q-1|_{2'}$ does not divide $\lambda_1\lambda_2$. Since $a\geq 2$, Theorem \ref{t: catalan} implies that $|q-1|_{2'}$ is non-trivial; hence we have a contradiction with Lemma \ref{l: lcm}.

Suppose next that $\lambda_1$ divides $2(\sqrt{q}\pm 1)$; then we conclude that $|\sqrt{q}\mp 1|_{2'}$ does not divide $\lambda_1\lambda_2$ and we deduce that $|\sqrt{q}\mp 1|_{2'}$ is trivial, i.e. $\sqrt{q}\mp 1=2^x$ for some positive integer $x$; if $\sqrt{q}>3$, then this implies in particular that $\frac{\sqrt{q}\pm1}{2}$ is odd. Note too that $\frac{q+1}{2}$ is odd and so, since $(\frac{\sqrt{q}\pm 1}{2}, \frac{q+1}{2})=1$, Lemma \ref{l: lcm} implies that $\lambda_1\in \{\frac{\sqrt{q}\pm 1}{2}, \sqrt{q}\pm 1, 2(\sqrt{q}\pm 1)\}$ and $\lambda_2\in\{\frac{q+1}{2}, q+1\}$. There are, therefore, twelve possibilities for $(\lambda_1, \lambda_2)$; in the following table we list them along with a polynomial $f$ which divides $\chi$ and which, since $\sqrt{q}>5$, is divisible by a prime other than $2$ and $p$; note that we write $y$ for $\sqrt{q}$.
\begin{center}
 \begin{tabular}{|c|c|c|}
  \hline
$\lambda_1$ & $\lambda_2$ & $f$ \\
\hline
$\frac{y+1}{2}$ & $\frac{y^2+1}{2}$ & $y^3-3y^2-3y-7$\\
$\frac{y-1}{2}$ & $\frac{y^2+1}{2}$ & $y^3-5y^2-3y-1$\\
$y+1$ & $\frac{y^2+1}{2}$ & $y^3-y^2-3y-5$ \\
$y-1$ & $\frac{y^2+1}{2}$ & $y^3-3y^2-3y+3$\\
$2(y+1)$ & $\frac{y^2+1}{2}$ & $y^3-3y+4$\\
$2(y-1)$ & $\frac{y^2+1}{2}$ & $y^3-2y^2-3y+2$\\
$\frac{y+1}{2}$ & $y^2+1$ & $y^3-3y^2-y-5$\\
$\frac{y-1}{2}$ & $y^2+1$ & $y^3-5y^2-y-3$\\
$y+1$ & $y^2+1$ & $y^3-y^2-y-3$\\
$y-1$ & $y^2+1$ & $y^3-3y^2-y-1$\\
$2(y+1)$ & $y^2+1$ & $y^3-y-2$\\
$2(y-1)$ & $y^2+1$ & $y^2-2y-1$\\
\hline
 \end{tabular}
\end{center}

We justify the first line of the above table, the others are similar. In this case
$$\chi= |S:T|q(q^2-1)\left(\frac{2}{\sqrt{q}+1} + \frac{2}{q+1} - \frac12\right) = -\frac12 q(\sqrt{q}-1)(\sqrt{q}^3-3q-3\sqrt{q}-7).$$
Now observe that $(\sqrt{q}^3-3q-3\sqrt{q}-7, q)\big|7$ and $(\sqrt{q}^3-3q-3\sqrt{q}-7, \sqrt{q}-1)\big|12$; since $\sqrt{q}-1$ is a power of $2$ in this case, Lemma \ref{l: lcm} implies that
$$\sqrt{q}^3-3q-3\sqrt{q}-7\leq 28$$
and so $\sqrt{q}\leq 5$ which is a contradiction, as required.

Now when $\sqrt{q}=3,5$ we consult \cite{atlas}. In the latter case we must have $\lambda_1\in\{13,26\}$ and $\lambda_2\in\{6,12\}$; checking these four combinations we find that only $\Lambda=\{6,13\}$ gives a valid value for $\chi$. When $\sqrt{q}=3$ we must have $\lambda_1\in\{5,10\}$ and $\lambda_2\in\{4,6,8\}$; checking these six combinations we find that only $\Lambda\in\{\{5,4\}, \{5,6\}, \{10,4\}\}$ give valid values for $\chi$. The result follows.
\end{proof}

\begin{lem}
 If $p=2$ then one of the following holds:
\begin{enumerate}
\item $S=SL_2(16).2$, $\Lambda=\{6,5\}$, $\chi=-2^6\cdot 17$;
\item $S=SL_2(16).2$, $\Lambda=\{10,3\} $, $\chi = -2^5\cdot 17$;
\end{enumerate}
\end{lem}
\begin{proof}
We assume that $q>4$ since $SL_2(4).2\cong PGL_2(5)$. Since $a_1=2$ we know that $\lambda_1$ divides one of $2(\sqrt{q}+1)$, $2(\sqrt{q}-1)$, $4$. If $a_2=2$ then the same can be said of $\lambda_2$; if $a_2=1$ then $\lambda_2$ divides one of $q-1$, $q+1$, $2$.

Now, since $\sqrt{q}+1, \sqrt{q}-1, q+1$ are pairwise coprime, Lemma \ref{l: lcm} implies that $\lambda_1$ is divisible by one of these, $\lambda_2$ is divisible by another, and the third is a power of a prime $s$ (which prime, in turn, divides $\chi$). In fact we know that $\lambda_1\in\{2(\sqrt{q}-1), 2(\sqrt{q}+1)\}$. 

Suppose first that $a_2=2$. Then clearly $\{\lambda_1, \lambda_2\}=\{2(\sqrt{q}+1), 2(\sqrt{q}-1)\}$. In this case $f=q-2\sqrt{q}-1$ divides $\chi$; furthermore, for $q>4$, $f$ is divisible by a prime other than $2$ and $s$ and this case is excluded. We conclude that $a_2=1$ and $\lambda_2$ divides one of $q-1$ and $q+1$, $2$

Suppose that $\lambda_2$ divides $q+1$; then $\lambda_2=q+1$. In this case there are two possibilities for $\lambda_1$. We list these possibilities in the table below, along with a polynomial $f$ which $\chi$ and which, whenever $\sqrt{q}>2$, is divisible by a prime other than $2$ and $s$ (thus these cases are excluded).
\begin{center}
 \begin{tabular}{|c|c|c|}
  \hline
$\lambda_1$ & $\lambda_2$ & $f$ \\
\hline
$2(\sqrt{q}+1)$ & $q+1$ & $(\sqrt{q})^3-\sqrt{q}-2$\\
$2(\sqrt{q}-1)$ & $q+1$ & $q-2\sqrt{q}-1$\\
\hline
 \end{tabular}
\end{center}

We are left with the possibility that $\lambda_2$ divides $q-1$ and that $q+1$ is a power of the prime $s$. Since $q$ is an even power of $2$, Theorem \ref{t: zsig} implies that $q+1=s$. We assume that $q>16$ and split into two cases.

First suppose that $\lambda_1=2(\sqrt{q}+1)$ and $\lambda_2=\frac{q-1}{x}$ for some $x$ dividing $\sqrt{q}+1$. Then
$$\chi = |S|(\frac{1}{2(\sqrt{q}+1)} + \frac{x}{q-1} - \frac{1}{2}) = -\frac{|S|}{2(q-1)}(q-\sqrt{q}-2x).$$
Let $f=q-\sqrt{q}-2x$; since $x$ is odd we know that $(f,q)=2$; we also know that $f<q+1$. We conclude that $f=2$; but $f\geq q-3\sqrt{q}-2>2$ for $q>16$, and we have a contradiction as required.

The other possibility is that $\lambda_1=2(\sqrt{q}-1)$ and $\lambda_2=\frac{q-1}{x}$ for some $x$ dividing $\sqrt{q}-1$. Then
$$\chi = |S|(\frac{1}{2(\sqrt{q}-1)} + \frac{x}{q-1} - \frac12) = \frac{|S|}{2(q-1)}(q-\sqrt{q}-2x-2).$$
Let $f=q-\sqrt{q}-2x-2$; since $x$ is odd we know that $(f,q)=4$; we also know that $f<q+1$. We conclude that $f=4$; but $f\geq q-3\sqrt{q}>4$ for $q>16$, and we have a contradiction as required.

We are left with the case when $q=16$ and $\lambda_2$ divides $q-1=15$. Then $\lambda_1\in\{6,10\}$ and the only possibilities are 
$$(\lambda_1, \lambda_2) \in\{ (6,5), (10,3), (6,15), (10, 15)\}.$$
Checking each in turn we exclude the last two cases and the result follows.
\end{proof}

\subsection{$T={^2B_2(q)}$}

In this situation $q=2^{2x+1}$ for some integer $x\geq 1$.
We refer to \cite{suzuki2}, in particular Proposition 1 and Theorem 9 of that paper, to conclude that any element in $T$ whose order is not a power of $2$ must have order dividing $q-1$, $q-r+1$ or $q+r+1$ where $r=2^{x+1}$. Note, moreover, that $q-1$, $q-r+1$ and $q+r+1$ are pairwise coprime.

Recall that the outer automorphism group of $T$ is isomorphic to the Galois group of $\Fq$, i.e. it consists of field automorphisms and is a group of odd order. Now write $S=\langle T, \delta \rangle$ where $\delta$ is a field automorphism of order $a$ where $a$ divides $2x+1$. Since $a$ is of odd order, for $S$ to be a $(2,m,n)$-group we must have two elements $g$ and $h$ such that $gT$ and $hT$ both have order $a$ in $S/T$.
We conclude that
$$\{o(g^a), o(h^a)\} \subset \{q-1, q-r+1, q+r+1\}.$$

\begin{lem}
 $S=T$.
\end{lem}
\begin{proof}
Suppose that $S=\langle T, \delta \rangle \subset \langle B_2(q), \delta \rangle$ where $\delta$ is non-trivial (and we have abused notation slightly by writing $\delta$ as a field automorphism of $B_2(q)$). Write $(t, \delta)\in B_2(q)\rtimes\langle \delta \rangle$ and observe that
$$u=(t,\delta)^a = t\cdot t^\delta \cdot t^{\delta^2}\cdots t^{\delta^{a-1}}.$$
In particular observe that $u^\delta = t^\delta\cdots t^{\delta^{a-1}}\cdot t = t^{-1}ut$ and we obtain that $u$ lies in a conjugacy class of $B_2(q)$ that is stable under $\delta$. Now we apply Proposition \ref{p: lang steinberg 2b2} and conclude that $u$ is conjugate in $B_2(q)$ to an element of $B_2(q_0)$ where $q=q_0^a$.

We apply this fact with $u$ equal to $g^a$ or $h^a$; in both cases $u$ is of odd order and we conclude that $o(u)$ divides $a(q_0^2-1)(q_0^4-1)$. For $a\geq 11$ we have
$$o(u)<q-r+1 = \min\{q-1, q-r+1, q+r+1\}$$
which is a contradiction. Thus we assume that $a\leq 9$.

If $a>4$, then Theorem \ref{t: zsig} implies that there is a prime dividing $q-1$ and a prime dividing $q^2+1$, neither of which divide $q_0^4-1$; what is more (since $3|2^2-1$, $5|2^4-1$ and $7|2^3-1$) both of these primes may be taken to be larger than $7$. We conclude that neither of these primes divide $o(g)\cdot o(h)$ and so both must divide $\chi$ which is a contradiction.

Finally suppose that $a=3$. Then Theorem \ref{t: zsig} implies that there is a prime greater than $3$ dividing $q^2+1$ which does not divide $q_0^4-1$; now $q_0^2+q_0+1$ divides $q-1$ and is coprime to $q_0^4-1$ thus there is a prime greater than $3$ dividing $q-1$ which does not divide $q_0^4-1$. Again we conclude that neither of these primes divide $o(g)\cdot o(h)$ and so both must divide $\chi$ which is a contradiction.
\end{proof}

\begin{lem}
If $T= {^2B_2(q)}$, then we have a contradiction.
\end{lem}
\begin{proof}
We know that $S=T$ and that $\Lambda\subset \{q-1, q-r+1, q+r+1\}$; we go through the possibilities in turn.

If $\Lambda=\{q-r+1, q+r+1\}$ then
\begin{equation*}
 \begin{aligned}
 \chi &= |S|(\frac1m + \frac 1n - \frac12) \\
&= q^2(q^2+1)(q-1)\left(\frac{1}{q-r+1}+ \frac1{q+r+1} - \frac12\right) \\
&=-\frac12q^2(q-1)(q^2-4q-3).
 \end{aligned}
\end{equation*}
Now $q^2-4q-3$ is odd thus we require that $(q-1)(q^2-4q-3)$ is a prime power. But $(q^2-4q-3, q-1)=1$ and we have a contradiction.

If $\Lambda=\{q-r+1, q-1\}$ then
\begin{equation*}
 \begin{aligned}
 \chi &= |S|(\frac1m + \frac 1n - \frac12) \\
&= q^2(q^2+1)(q-1)\left(\frac{1}{q-r+1}+ \frac1{q-1} - \frac12\right) \\
&=-\frac12q^2(q+r+1)(q^2-qr-4q+3r-1).
 \end{aligned}
\end{equation*}
Now $q^2-qr-4q+3r-1$ is odd thus we require that $(q+r+1)(q^2-qr-4q+3r-1)$ is a prime power. But $(q+r+1, q^2-qr-4q+3r-1)<q+r+1$ and we have a contradiction.

If $\Lambda=\{q+r+1, q-1\}$ then
\begin{equation*}
 \begin{aligned}
 \chi &= |S|(\frac1m + \frac 1n - \frac12) \\
&= q^2(q^2+1)(q-1)\left(\frac{1}{q+r+1}+ \frac1{q-1} - \frac12\right) \\
&=-\frac12q^2(q-r+1)(q^2+qr-4q-3r-1)
 \end{aligned}
\end{equation*}
Now $(q^2+qr-4q-2r-1)$ is odd thus we require that $(q-r+1)(q^2+qr-4q-3r-1)$ is a prime power. But $(q-r+1, q^2+qr-4q-3r-1)<q-r+1$ and we have a contradiction.
\end{proof}

\subsection{$T=PSL_3(q)$}

Throughout this section, we assume that $q>2$ (since $PSL_3(2)\cong PSL_2(7)$). We start with an easy result:
\begin{lem}
There exists $a\in\mathbb{Z}^+$ such that $q-1=r^a$ for some prime $r$.
\end{lem}
\begin{proof}
Observe first that the Sylow $t$-subgroups are non-cyclic for $t=p$ and $t|q-1$. Thus, if $t$ is a prime such that $t|p(q-1)$, then $t|\chi$ and we conclude that $q-1$ is a prime power.
\end{proof}

This result, along with Theorem \ref{t: catalan} implies that, for $q\neq 4$, the group $SL_3(q)$ has trivial centre; thus $T=SL_3(q)$. 

\begin{cor}\label{c: graph}
If $q$ is odd and $q\neq 9$, then $q$ is prime and $T\leq S\leq \langle T, \sigma\rangle$ where $\sigma$ is a graph automorphism of $T$. If $q=2^a$ for some integer $a$, then $q-1$ is prime; what is more, if $q\neq 4$, $T\leq S \leq \langle T, \delta, \sigma\rangle$ where $\sigma$ is a graph automorphism of $T$ and $\delta$ is a field automorphism of $T$ of order $a$.\end{cor}
\begin{proof}
Suppose first that $q$ is odd. Since $q=2^a+1$ we know that $T$ admits no diagonal outer automorphisms. Now Theorem \ref{t: catalan} implies that, unless $q\neq 9$, $q=p$. Thus $T$ admits no field outer automorphisms, and the result follows.

Now suppose that $q=2^a$; then $q-1$ is a prime power, and Theorem \ref{t: catalan} implies that $q-1$ is prime. Thus, for $q\neq 4$, $T$ admits no diagonal automorphisms, and the result follows.
\end{proof}

In what follows we fix $\sigma$ to be the graph automorphism of $T$ where, for $t\in T$, $t^\sigma = t^{-T}$ (the inverse transpose of $t$). Before we give a classification of maps we give a group-theoretic lemma:

\begin{lem}\label{l: john}
Suppose $S=\langle T, \sigma \rangle$ with $q>4$. All elements in $S$ of order divisible by $q^2+q+1$ have order $q^2+q+1$. All elements in $S$ of order divisible by $\frac{q+1}{(2, q+1)}$ have order dividing $q^2-1$. If $q$ is even, then all elements in $S\backslash T$ of order divisible by $q+1$ have order dividing $2(q+1)$.
\end{lem}
\begin{proof}
Since $q>4$ we have $T=SL_3(q)$. Since $q=2^a$ or $2^a+1$ we know that $r_2$ and $r_3$ exist. Let $x$ (resp. $y$) be elements of these orders in $T$.

Observe that $x$ is diagonalizable in $\Fqtwo$ but not in $\Fq$; we conclude that the eigenvalues of $x$ are equal to $\lambda, \lambda^q, \mu$ for some $\lambda, \mu\in \Fqtwo$. If $\lambda=\lambda^q$ then $\lambda\in\Fq$ and so $\mu\in\Fq$ which is a contradiction. If $\lambda=\mu$ then the determinant of $x$ is equal to $\lambda^{q+2}$; since $\lambda\in \Fqtwo^*$ and the determinant of $x$ equals $1$ we conclude that $\lambda^3=1$. But then $r_2=3$ and $\lambda^q=\lambda=\mu$ which is, again, a contradiction. We conclude that all eigenvalues of $x$ are distinct.

Similarly $y$ is diagonalizable in $\Fqthree$ but not in $\Fq$; thus the eigenvalues of $y$ are equal to $\xi, \xi^2, \xi^3$ for some $\xi\in\Fqthree$. If these eigenvalues are not distinct then $\xi$ lies either in $\Fqtwo$ or in $\Fq$ and in both cases we have a contradiction. 

Thus the eigenvalues of both elements, $x$ and $y$, are distinct, i.e. $x$ and $y$ are regular semisimple and, in paricular, both $C_T(x)$ and $C_T(y)$ are maximal tori in $PSL_3(q)$. It follows immediately that $C_T(x)$ (resp. $C_T(y)$) is cyclic of order $q^2-1$ (resp. of order $q^2+q+1$). Now Sylow arguments ensure that all cyclic groups $C_{q^2-1}$ (resp. $C_{q^2+q+1}$) are conjugate to each other; furthermore $|N_T(C_{q^2+q+1}): C_{q^2+q+1}|=3$ and $|N_T(C_{q^2-1}): C_{q^2-1}|=2$.

Let us prove first that an element in $S$ of order divisible by $q^2+q+1$ has order $q^2+q+1$; the previous paragraph implies that the statement is true for elements in $T$ so we must consider elements in $S\backslash T$. Let $h\in T$ be an element of order $q^2+q+1$. Since $|S:T|=2$ and $|N_T(C_{q^2+q+1}): C_{q^2+q+1}|=3$ it is sufficient to prove that there exists $g\in S\backslash T$ such that $|\langle h, g\rangle|=2(q^2+q+1)$ and $\langle h, g\rangle$ is not cyclic.

Now we use the standard fact that $h^{-T}$ is conjugate in $GL_3(q)$ to $h^{-1}$. Since $h^{-1}$ is an element of order $q^2+q+1$, the argument above implies that $C_{GL_3(q)}(h^{-1})$ is a maximal torus; indeed $C_{GL_3(q)}(h^{-1}\cong C_{q^3-1}$. Now we appeal to \cite[(4.3.13)]{kl} to conclude that $GL_1(q^3)$ intersects every coset of $SL_3(q)$ in $GL_3(q)$; in other words the conjugacy class of $h^{-1}$ in $GL_3(q)$ does not split when we restrict to $SL_3(q)$. Thus, in particular, there exists $g_0\in SL_3(q)$ such that $g_0h^{-T}g_0^{-1}=h^{-1}$; consequently there exists $g (=g_0\sigma)$ in $S\backslash T$ such that $ghg^{-1}=h^{-1}$. Since $\langle g, h\rangle$ is dihedral we are done.

Next we let $f\in S$ be an element in $S$ of order divisible by $\frac{q+1}{(q+1,2)}$ (and hence divisible by $r_2$); we know, by the centralizer arguments above, that if $f\in T$ then $f$ has order dividing $q^2-1$. Thus assume that $f\in S\backslash T$; then $k=f^2\in T$ and $k$ has order divisible by $\frac{q+1}{(q+1,2)}$ and so, once again, $k$ has order dividing $q^2-1$. 

Clearly, then, $f$ has order dividing $2(q^2-1)$. If $q$ is odd, then, since $q-1=2^a$ we conclude that either $o(f)$ divides $q^2-1$ or else $o(f)=2(q^2-1)$. If $q$ is even, then, since $q-1$ is an odd prime, either $o(f)$ divides $2(q+1)$ or else $o(f)=2(q^2-1)$. Thus in both cases to prove the result it suffices to show that $S$ does not contain an element of order $2(q^2-1)$.

First we construct $g\in S \backslash T$ such that $\langle k,g\rangle$ is dihedral of order $2o(k)$. This time things are easier, since $k$ lies in $K<SL_3(q)$ with $K\cong GL_2(q)$ and such that $\sigma$ normalizes $K$ and takes every element of $K$ to its inverse transpose. As above there exists $g_0$ in $K$ such that $g_0k^{-T}g_0^{-1}=k^{-1}$ and, setting $g=g_0\sigma$ we obtain $g\in S\backslash T$ such that $gkg^{-1}=k^{-1}$ and so $\langle g, k\rangle$ is dihedral. 

To show that no element of order $2(q^2-1)$ exists in $S\backslash T$ we must be sure that the element $k$  is not real, i.e. we must show that $k$ is not conjugate to its inverse in $SL_3(q)$. Observe that the eigenvalues of $k$ are $\{\lambda_1, \lambda_2, \lambda_3\}$ where $\lambda_1, \lambda_2\in \mathbb{F}_{q^2}\backslash \mathbb{F}_q$ and $\lambda_3\in \Fq$ has multiplicative order equal to $q-1$. Now classical results guarantee that $k$ is not real (see, for instance, \cite{gillsingh}) and the result follows.
\end{proof}

To make the exposition clearer from this point on, we split into odd and even characteristic cases.

\begin{lem}\label{l: psl3 odd}
Suppose that $T=PSL_3(q)$ and $q=2^a+1$. Then one of the following holds:
\begin{enumerate}
 \item $T=SL_3(3)$, $\Lambda=\{4,13\}$, $\chi=-|S:T|2^2\cdot 3^5$;
 \item $S=SL_3(3)$, $\Lambda=\{13,13\}$, $\chi=-2^3\cdot 3^5$;
 \item $S=SL_3(5)$, $\Lambda=\{3,31\}$, $\chi=-2^4\cdot 5^5$;
\end{enumerate}
\end{lem}
\begin{proof}
Suppose first that $q>3$. The two primes dividing $\chi$ must be $2$ and $p$, hence $\Lambda$ must contain elements divisible by $r_2$ and $r_3$. (Note that, since $q>3$, both of these exist and both are odd.) Lemma \ref{l: john} implies that no element exists that is divisible by both $r_2$ and $r_3$ so we assume that $r_3$ divides $\lambda_1$. Observe that
$$(q^2+q+1, q-1) = (q^2+q+1, q) = (q^2+q+1, q+1)=1$$
and we conclude that, in fact, $q^2+q+1$ must divide $\lambda_1$. Now Lemma \ref{l: john} implies that, for $q\neq 9$, $\lambda_1=q^2+q+1$; \cite{atlas} implies that, for $q=9$, we also have $\lambda_1=q^2+q+1$.

Similarly we conclude that $\frac{q+1}2$ divides $\lambda_2$; then Lemma \ref{l: john} implies that, for $q\neq 9$, $\lambda_2=\frac{q^2-1}{x}$ for some $x=2^b|2(q-1)$; \cite{atlas} implies that, for $q=9$, we have the same.

Now we calculate $\chi$:
\begin{equation*}
 \begin{aligned}
  \chi &= |S|(\frac1m + \frac 1n - \frac12) \\
&= |S:T|q^3(q^2-1)(q^3-1)\left(\frac{x}{q^2-1}+ \frac1{q^2+q+1} - \frac12\right) \\
&=-\frac12|S:T|q^3(q-1)\big(q^4+q^3-(2x+2)q^2-(2x+1)q-(2x-1)\big) 
 \end{aligned}
\end{equation*}
Setting $f=q^4+q^3-(2x+2)q^2-(2x+1)q-(2x-1)$, observe that $f\equiv -2x+1\not\equiv0\mod q$ for $x<\frac{q-1}{2}$; similarly $f\equiv -6x\not\equiv 0 \mod(q-1)$ for $x<\frac{q-1}{2}$. Thus if $x\leq \frac{q-1}{4}$ we must have $|f|<q(q-1)$ which implies that $q< 5$.

If $x=\frac{q-1}{2}$ we have 
$$f=q^4-2q^2-q+2 = (q-1)(q^3+q^2-q-2).$$
Setting $g=q^3+q^2-q-2$, observe that $g\equiv 2\not\equiv0\mod q$ for $q>2$; similarly $g\equiv -1\not\equiv 0 \mod(q-1)$ for all $q$. Thus, since $q>3$, we must have $|g|<q(q-1)$ which implies that $q<5$, a contradiction.

If $x=q-1$ we have 
$$f=q^4-q^3-2q^2-q+3 = (q-1)(q^3-2q-3).$$
Setting $g=q^3-2q+3$, observe that $g\equiv -3\not\equiv0\mod q$ for $q>3$; similarly $g\equiv -4\not\equiv 0 \mod(q-1)$ for $q>5$. Thus, for $q>5$, we must have $|g|<q(q-1)$ which implies that $q<3$, a contradiction.

If $x=2(q-1)$ we have 
$$f=q^4-3q^3-2q^2-q+5 = (q-1)(q^3-2q^2-4q-5).$$
Setting $g=q^3-2q^2-4q-5$, observe that $g\equiv -5\not\equiv0\mod q$ for $q>5$; similarly $g\equiv -10\not\equiv 0 \mod(q-1)$ for $q>5$. Thus, for $q>5$, we must have $|g|<q(q-1)$ which implies that $q<5$, a contradiction.

We are left with the possibility that $q=3$ or $q=5$. When $q=5$ we must have $\Lambda=\{31, \lambda_1\}$ and $\lambda_1=3,6,12$ or $24$. Calculating the Euler characteristic in each case we find that $\chi = -|S:T|5^5\cdot 2^4, -|S:T|5^3\cdot2^7\cdot 7, -|S:T|5^3\cdot 2^3\cdot 11\cdot13$, $-|S:T|5^3\cdot 2^2\cdot 317$ respectively. We conclude that $\lambda_2=3$; since there are no elements of order $31$ or $3$ in $Aut T\backslash T$ we conclude that $S=T$ in this situation, as required.

When $q=3$, $|T|=2^4\cdot 3^3\cdot 13$ and we conclude that $\Lambda=\{13, \lambda_1\}$ where $\lambda_1$ ranges through the element orders ($>2$) of elements in $S$. Using \cite{atlas}, we go through these one at a time:
\begin{center}
 \begin{tabular}{|c|c|}
\hline
  $\lambda_1$ & Prime dividing $\chi$ or $(S, \chi)$\\
\hline
$3$ & $7$\\
$4$ & $(T,-2^2\cdot 3^5)$ or $(T.2, -2^3\cdot 3^5)$\\
$6$ & $5$\\
$8$ & $31$\\
$12$ & $53$ \\
$13$ & $(T, -2^3\cdot3^5)$ \\
\hline
 \end{tabular}
\end{center}
The result follows.
\end{proof}

From now on we have $q=2^a$ for some $a\geq 1$. In this case we must account for outer automorphisms that are not just graph automorphisms; write $\delta$ for a field automorphism of $T$ of order $a$. In what follows we study an automorphism $\zeta=\sigma^x\delta^y$ where $x$ and $y$ satisfy $0\leq x \leq 1; 1\leq y\leq a-1$. Note, in particular, that $\zeta$ can be extended to an automorphism of $SL_3(\overline{\Fq})$ and, in this situation, it has a finite number of fixed points.

\begin{lem}\label{l: john 2}
 Suppose that $q=2^a$ with $a>2$. Any element in $S$ of order divisible by $q^2+q+1$ has order equal to $q^2+q+1$. Any element in $S\backslash T$ of order divisible by $q+1$ has order dividing $2(q+1)$.
\end{lem}
\begin{proof}
 Any element of $S$ lies in a cyclic extension of $T$; if an element lies in an extension $\langle T, \sigma\rangle$ where $\sigma$ is a graph automorphism then Lemma \ref{l: john} gives the result. Thus assume this is not the case and consider elements in an extension of form $\langle T,\zeta\rangle$ where $\zeta$ is given above.

Let $x$ be the order of $\zeta$; thus $x$ divides $2a$, and consider an element $(t,\zeta)\in T\rtimes \langle \zeta\rangle$. Observe that
$$u=(t,\zeta)^x = g\cdot t \cdot t^2 \cdots g^{\zeta^{x-1}}.$$
In particular observe that $u^\zeta = t^{-1}ut$ and we obtain that $u$ lies in a conjugacy class that is stable under $\zeta$. Now we apply Proposition \ref{p: lang steinberg}, and conclude that $u$ is conjugate to an element of $SL_3(r)$ where $q=r^a$. Observe in particular that $|o(h)|_{2'}\leq r^2+r+1$.

Suppose that $g=(t,\zeta)$ is an element of order divisible by $q^2+q+1$. Then the order of $|o(g)|_{2'}$ divides $a(r^2+r+1)$ and so we have that
$q^2+q+1\leq a(r^2+r+1)$ which is a contradiction for $a>1$. Thus any element of $S$ of order divisible by $q^2+q+1$ must lie in $T$ and so has order equal to $q^2+q+1$.

Next suppose that $g=(t,\zeta)$ is an element of $S$ of order divisible by $q+1$. Then, as before, we have that $q+1\leq a(r^2+r+1)$. If $a \geq 4$ this implies that $q=16$ or $32$. Then $o(g)$ is divisible by $11$ or $17$ which does not divide $a|SL_3(r)|$ and we are done. If $a=3$ then we have $q<64$ and we conclude that $q=8$; now \cite{atlas} confirms the result. Finally suppose that $a=2$; then we must have $q+1$ dividing $|SL_3(\sqrt{q}|$. Now $q+1$ is coprime with both $q-1$ and $q+\sqrt{q}+1$, and we have a contradiction.
\end{proof}

\begin{lem}\label{psl3 even}
 Suppose that $T=PSL_3(q)$ and $q=2^a$ for some integer $a\geq 2$. Then one of the following holds:
\begin{enumerate}
 \item $S=PSL_3(4).2$, $\Lambda=\{5,14\}$, $\chi= -2^{10}\cdot 3^2$;
 \item $S=PSL_3(4).2$, $\Lambda=\{10,7\}$, $\chi= -2^7\cdot 3^4$;
 \item $S=PSL_3(4).3$, $\Lambda=\{15,21\}$, $\chi= -2^5\cdot 3^6$;
\end{enumerate}
\end{lem}
\begin{proof}
The two primes dividing $\chi$ must be $2$ and $q-1$, hence $\Lambda$ must contain elements divisible by $r_2$ and $r_3$. (Note that both of these exist.) Now Lemmas \ref{l: john} and \ref{l: john 2} imply that $\Lambda=\{q^2+q+1, \lambda_2\}$ where $\lambda_2\in\{q+1, 2(q+1),  q^2-1\}$. 

We start by assuming $q>4$ and we calculate $\chi$ for each of the three possibilities. First suppose $\lambda_2=q+1$:
\begin{equation*}
 \begin{aligned}
  \chi &= |S|(\frac1m + \frac 1n - \frac12) \\
&= |S:T|q^3(q^2-1)(q^3-1)\left(\frac{1}{q+1}+ \frac1{q^2+q+1} - \frac12\right) \\
&=-\frac12|S:T|q^3(q-1)^2\big(q^3-2q-3\big) 
 \end{aligned}
\end{equation*}
Setting $f=q^3-2q-3$, observe that $f\equiv -3\not\equiv0\mod q$; similarly $f\equiv -4\not\equiv 0 \mod(q-1)$. Thus we must have $|f|<q(q-1)$ which implies that $q\leq 4$, a contradiction.

Next suppose that $\lambda_2=q^2-1$:
\begin{equation*}
 \begin{aligned}
  \chi &= |S|(\frac1m + \frac 1n - \frac12) \\
&= |S:T|q^3(q^2-1)(q^3-1)\left(\frac{1}{q^2-1}+ \frac1{q^2+q+1} - \frac12\right) \\
&=-\frac12|S:T|q^3(q-1)\big(q^4+q^3-4q^2-3q-1\big) 
 \end{aligned}
\end{equation*}
Setting $f=q^4+q^3-4q^2-3q-1$, observe that $f\equiv -1\not\equiv0\mod q$; similarly $f\equiv -6\not\equiv 0 \mod(q-1)$. Thus we must have $|f|<q(q-1)$ which implies that $q\leq 4$, a contradiction.

Finally suppose that $\lambda_2=2(q+1)$:
\begin{equation*}
 \begin{aligned}
  \chi &= |S|(\frac1m + \frac 1n - \frac12) \\
&= |S:T|q^3(q^2-1)(q^3-1)\left(\frac{1}{2(q+1)}+ \frac1{q^2+q+1} - \frac12\right) \\
&=-\frac12|S:T|q^3(q-1)^2\big(q^3+q^2-q-2\big) 
 \end{aligned}
\end{equation*}
Setting $f=q^3+q^2-q-2$, observe that $f\equiv -2\not\equiv0\mod q$; similarly $f\equiv -1\not\equiv 0 \mod(q-1)$. Thus we must have $|f|<q(q-1)$ which implies that $q\leq 4$, a contradiction.

We are left with the possibility that $q=4$; in this case we know that $s=3$ and we must have $\Lambda=\{\lambda_1, \lambda_2\}$ with $5|\lambda_1$ and $7|\lambda_2$. Consulting \cite{atlas} we see that $\lambda_1\in\{5,10,15\}$ and $\lambda_2\in\{7,14,21\}$. Now we go through the nine possibilities; all cases but three may be excluded:

\begin{center}
 \begin{tabular}{|c|c|| c | c || c| c|}
\hline
$\Lambda$ & Prime dividing $\chi$ & $\Lambda$ & Prime dividing $\chi$ & $\Lambda$ & Prime dividing $\chi$ \\
\hline
\{5,7\}  & 11 & \{5,14\} & * & \{5,21\} & 53 \\
\{10, 7\} & * & \{10, 14\} & 23 & \{10, 21\} & 37 \\
\{15,7\} & 61 & \{15, 14\} & 19 & \{15, 21\} & * \\
\hline
 \end{tabular}
\end{center}

Note that the outer automorphism group of $PSL_3(4)$ has order $12$. If $\Lambda=\{5,14\}$ or $\{10, 7\}$ then $(\lambda, 12) \leq 2$ for $\lambda\in \Lambda$. What is more in both $PSL_3(4)$ does not contain elements of order $10$ nor of order $14$; thus, in both cases we must generate a degree 2 extension, $T.2$. If $\Lambda=\{15,21\}$ then $(\lambda, 12) =3$ for all $\lambda \in \Lambda$; furthermore there are no elements of order $15$ nor of order $21$ in $PSL_3(4)$. Thus, since $(ghT)^2=1$ we conclude that $gT = (hT)^{-1}\in S/T$; thus $g$ and $h$ generate a degree 3 extension, $T.3$. The result follows.
\end{proof}

\subsection{$T=PSU_3(q)$}

In this section we proceed very similarly to the previous, although we have the happy advantage that there are no graph automorphisms for $T$. Note that we assume throughout that $q>2$ (since $PSU_3(2)$ is solvable). We start with an easy result:

\begin{lem}
There exists $a\in\mathbb{Z}^+$ such that $q+1=r^a$ for some prime $r$.
\end{lem}
\begin{proof}
Observe first that the Sylow $t$-subgroups are non-cyclic for $t=p$ and $t|q+1$. Thus, if $t$ is a prime such that $t|p(q+1)$, then $t|\chi$ and we conclude that $q-1$ is a prime power.
\end{proof}

Once again we use Theorem \ref{t: catalan} to limit the possibilities.

\begin{cor}\label{c: psu3}
One of the following holds:
\begin{enumerate}
 \item $q$ is an odd prime and $S=T$ or $S=\langle T, \delta \rangle$ for some field automorphism $\delta$ of order $2$;
\item $q=2^a$ for some positive integer $a\neq 1, 3$ and $S\leq \langle T, \delta \rangle$ for some field automorphism $\delta$;
\item $q=8$.
\end{enumerate}
\end{cor}
\begin{proof}
Suppose first that $q$ is odd. Since $q=2^a-1$ we know that $T$ admits no diagonal outer automorphisms. Now Theorem \ref{t: catalan} implies that $q$ is prime; thus $T$ has an outer automorphism group of size $2$, and the result follows.

Now suppose that $q=2^a$ with $a\neq 1, 3$; then $q+1$ is a prime power and Theorem \ref{t: catalan} implies that $q+1$ is a prime. In particular $q+1$ is not divisible by $3$ and so $T$ admits no diagonal automorphisms; the result follows.
\end{proof}

Let us deal with the last situation first.

\begin{lem}
If $T=PSU_3(8)$ then $S=T$, $\Lambda=\{7,19\}$ and $\chi=-2^8\cdot 3^8$.
\end{lem}
\begin{proof}
Observe first that $\pi(T)=\{2,3,7,19\}$ and that $\pi_{nc}(T)=\{2,3\}$. Thus $\Lambda$ must contain elements divisible by $7$ and $19$. Consulting \cite{atlas} for almost simple groups $S$ with socle $T$ we see that the possible element orders are 7, 14, 21, 63, 19 and 57. Now we go through the eight possibilities; all cases but one may be excluded:

\begin{center}
\begin{tabular}{|c|c|| c | c |}
\hline
$\Lambda$ & Prime dividing $\chi$ or $(S, \chi)$ & $\Lambda$ & Prime dividing $\chi$ or $(S, \chi)$ \\
\hline
\{7,19\} & $(T, -2^8\cdot 3^8)$ & \{7, 57\} & 271 \\
\{14,19\}  & 5 & \{14,57\} & 139 \\
\{21,19\} & 11 & \{21,57\} & 347 \\
\{63,19\} & 1033 & \{63,57\} & 1117 \\
\hline
\end{tabular}
\end{center}
The result follows.
\end{proof}

We are interested in the order of elements in $S\backslash T$; we will need to use the Lang-Steinberg theorem in much the same way as we have already seen it with $T={^2B_2(q)}$ and $T=PSL_3(q)$.

\begin{lem}\label{l: john 3}
 Suppose that we are in one of the first two situations of Cor. \ref{c: psu3} and that $q>3$. Any element in $S$ of order divisible by $q^2-q+1$ has order equal to $q^2-q+1$. Any element in $S$ of order divisible by $\frac{q-1}{(2,q-1)}$ has order dividing $\frac{4}{(2,q-1)}(q^2-1)$.
\end{lem}
\begin{proof}
Suppose that $x$ (resp. $y$) is an element of order divisible by an odd prime dividing $q-1$ (resp. by  $r_3$). Since $T=SU_3(q)< SL_3(q^2)$ we know that every element of $T$ is diagonalizable (in $GL_3(q)$) over a field of order $q^2, q^4$ or $q^6$; furthermore, elements in $T$ that are diagonalizable over $\Fqtwo$ have order dividing $q+1$, thus $x$ and $y$ are not diagonaliable over $\Fqtwo$.

Proceeding now as per the proof of Lemma \ref{l: john} we conclude that both $x$ and $y$ have distinct eigenvalues and thus their centralizers are both maximal tori of $T$; in particular $C_T(x)\cong C_{q^2-1}$ and $C_T(y)\cong C_{q^2-q+1}$. We conclude, in particular, that any element in $T$ of order divisible by $q^2-q+1$ has order equal to $q^2-q+1$; similarly any element in $T$ of order divisible by $\frac{q-1}{(2,q-1)}$ has order dividing $q^2-1$.

Suppose, next, that $S=\langle T, \delta \rangle$ where $\delta$ is a field automorphism of order $x>1$. Let $(t, \delta)$ be an element of $S$; proceeding as per the proof of Lemma \ref{l: john 2}, and using Propostion \ref{p: lang steinberg} we conclude that $(t, \delta)^x$ lies in $SL_3(q_1^2)$ where $q_1^t=q$. Thus $(t, \delta)$ has order $xv$ where $v$ is the order of an element in $SL_3(q_1^2)$.

Suppose first that $(t, \delta)$ has order divisible by $q^2-q+1$. If $x$ is even then we conclude that $q^2-q+1$ has order $xv$ where $v$ is the order of an element in $SL_3(q)$. But now observe that $(q^2-q+1, |SL_3(q)|)=1$ and we have a contradiction.

If $x$ is odd then $q_1^{2x}-q_1^x+1 = xv$ where $v$ is the order of an element in $SL_3(q_1^2)$. Since $|v|_{p'}\leq q_1^4+q_1^2+1$ we have
$$q_1^{2x}-q_1^x+1 \leq x(q_1^4+q_1^2+1)$$
which is a contradiction unless $q_1=2$ and $x=3$. But in this case $q=8$ which is excluded.

Suppose next that $(g, \delta)$ has order divisible by $\frac{q-1}{(2,q-1)}$. If $\delta$ has order at most $2$ then the result follows immediately; in particular the result is true for $q$ odd and we suppose that $q$ is even. Then, as above, we have that $q_1^x-1$ divides $xv$ where $v$ is the order of an element in $SL_3(q_1^2)$. This implies immediately that $q_1^x-1 \leq x(q_1^4+q_1^2+1)$
which is a contradiction for $x\geq 8$. 

For $x=6,7$ we conclude that $q_1=2$. Since $2^7-1$ is a prime we know that it does not divide the order of an element of $SL_3(4)$ so we exclude $x=7$. If $x=6$, $q=2^x+1=65$ which is not a prime and so is excluded. If $x=5$ then $q_1^5-1$ divides either $5(q_1^4-1)$ or $5(q_1^4+q_1^2+1)$ which is a contradiction. If $x=4$ then $q_1^4-1$ divides either $4(q_1^4-1)$ or $4(q_1^4+q_1^2+1)$; the latter is impossible, the former gives the result. 

We are left with $x=3$. In this case $q_1^3-1$ divides either $3(q_1^4-1)$ (impossible) or $3(q_1^4+q_1^2+1)$; this in turn implies that $q_1^3-1$ divides $3(q_1^2+q_1+1)$ and we conclude that $q_1\leq 4$. If $q_1=2$ then $q=8$ which is excluded; if $q_1=4$ then the order of $(g, \delta)$ divides $q_1^3-1$ as required.
\end{proof}

\begin{lem}
If $q$ is odd, then $PSU_3(3)=T\leq S \leq PSU_3(3).2$ and one of the following holds:
\begin{enumerate}
 \item $S=T$, $\Lambda=\{3,7\}$, $\chi=-2^4\cdot 3^2$;
 \item $\Lambda=\{4,7\}$, $\chi=-|S:T|2^3\cdot 3^4$;
 \item $\Lambda=\{6,7\}$, $\chi=-|S:T|2^7\cdot 3^2$;
 \item $S=T$, $\Lambda=\{7,7\}$, $\chi=-2^4\cdot 3^4$;
\end{enumerate}
\end{lem}
\begin{proof}
 Suppose first that $q>3$. The two primes dividing $\chi$ must be $2$ and $p$ hence, writing $\Lambda=\{\lambda_1, \lambda_2\}$ we must have $(\frac{q-1}{2})(q^2+1+1)$ dividing $\lambda_1\lambda_2$.  Now Lemma \ref{l: john 3} implies that $\Lambda=\{q^2+q+1, \lambda_2\}$ where $\lambda_2=\frac{2(q^2-1)}{x}$ for some $x=2^b|4(q+1)$. Now we calculate $\chi$:
\begin{equation*}
 \begin{aligned}
  \chi &= |S|(\frac1m + \frac 1n - \frac12) \\
&= |S:T|q^3(q^2-1)(q^3+1)\left(\frac{x}{2(q^2-1)}+ \frac1{q^2-q+1} - \frac12\right) \\
&=-\frac12|S:T|q^3(q+1)\big(q^4-q^3-(x+2)q^2+(x+1)q-(x-1)\big) 
 \end{aligned}
\end{equation*}
Setting $f=q^4-q^3-(x+2)q^2+(x+1)q-(x-1)$, observe that $f\equiv -(x-1)\not\equiv0\mod q$ for $x<q+1$; similarly $f\equiv -3x\not\equiv 0 \mod(q+1)$ for $x<q+1$. Thus if $x\leq \frac{q+1}{2}$ we must have $|f|<q(q+1)$ which implies that $q<7$, a contradiction.

If $x=q+1$ we have 
$$f=q^4-2q^3-2q^2+q = q(q+1)(q^2-3q+1).$$
Setting $g=q^2-3q+1$, observe that $(g,q)=1$ and $(g, q+1)\leq 5$. Thus $q^2-3q+1\leq 5$ which is a contradiction.

If $x=2(q+1)$ we have 
$$f=q^4-3q^3-2q^2+q-1 = (q+1)(q^3-4q^2+2q-1).$$
Setting $g=q^3-4q^2+2q-1$, observe that $g\equiv -1\not\equiv0\mod q$; similarly $g\equiv -8\not\equiv 0 \mod(q-1)$ for $q>7$. Thus, for $q>7$, we must have $|g|<q(q-1)$ which implies that $q<5$, a contradiction. If $q=7$ then $f$ is divisible by $5$ which is a contradiction.

If $x=4(q+1)$ we have 
$$f=q^4-5q^3-2q^2+q-3 = (q+1)(q^3-6q^2+4q-3).$$
Setting $g=q^3-6q^2+4q-3$, observe that $g\equiv -3\not\equiv0\mod q$ for $q>3$; similarly $g\equiv -14\not\equiv 0 \mod(q+1)$ for all $q\neq 13$ (and we know that $q\neq 13$ since $q+1$ is a power of $2$). Thus, for $q>3$, we must have $|g|<q(q+1)$ which implies that $q<7$, a contradiction.

We are left with the possibility that $q=3$. Then $|T|=2^5\cdot 3^3\cdot 7$ and we conclude that $\Lambda=\{7, \lambda_1\}$ where $\lambda_1$ ranges through the element orders ($>2$) of elements in $S$. Using \cite{atlas}, we go through these one at a time:
\begin{center}
 \begin{tabular}{|c|c|}
\hline
  $\lambda_1$ & Prime dividing $\chi$ or $(S, \chi)$\\
\hline
$3$ & $(T,-2^4\cdot 3^2)$ \\
$4$ & $(T,-2^3\cdot 3^4)$ or $(T.2, -2^4\cdot 3^4)$\\
$6$ & $(T,-2^7\cdot 3^2)$ or $(T.2, -2^8\cdot 3^2)$\\
$7$ & $(T, -2^4\cdot 3^4)$\\
$8$ & $13$ \\
$12$ & $23$ \\
\hline
 \end{tabular}
\end{center}
The result follows.
\end{proof}

\begin{lem}
 Suppose that $q=2^a$ with $a$ a positive integer not equal to $3$. Then
\begin{enumerate}
 \item $S=PSU_3(4).2$, $\Lambda=\{6,13\}$, $\chi=-2^8\cdot 5^3$;
\end{enumerate}
\end{lem}
\begin{proof}
 Assume first that $q>4$; hence, in particular, $q\geq 16$. The two primes dividing $\chi$ must be $2$ and $q+1$, hence, writing $\Lambda=\{\lambda_1, \lambda_2\}$, we must have $(\frac{q-1}{2})(q^2+1+1)$ dividing $\lambda_1\lambda_2$.  Now Lemma \ref{l: john 3} implies that $\Lambda=\{q^2+q+1, \lambda_2\}$ where $\lambda_2=\frac{2(q^2-1)}{x}$ for some $x=2^b|4(q+1)$. Now we calculate $\chi$:
\begin{equation*}
 \begin{aligned}
  \chi &= |S|(\frac1m + \frac 1n - \frac12) \\
&= |S:T|q^3(q^2-1)(q^3+1)\left(\frac{x}{4(q^2-1)}+ \frac1{q^2-q+1} - \frac12\right) \\
&=-\frac14|S:T|q^3(q+1)\big(2q^4-2q^3-(x+4)q^2+(x+2)q-(x-2)\big) 
 \end{aligned}
\end{equation*}
Setting $f=2q^4-2q^3-(x+4)q^2+(x+2)q-(x-2)$, observe that $f\equiv -(x-2)\not\equiv0\mod q$ for $x\leq q+1$; similarly $f\equiv -3x\not\equiv 0 \mod(q+1)$ for $x<q+1$. Thus if $x<q+1$ we must have $|f|<q(q+1)$ which implies that $q<16$, a contradiction. 

Now suppose that $x\geq q+1$; this implies that $\lambda_2\in\{q-1, 2(q-1), 4(q-1)\}$ and we go through these in turn.

If $\lambda_2=q-1$ then we have
\begin{equation*}
 \begin{aligned}
  \chi &= |S|(\frac1m + \frac 1n - \frac12) \\
&= |S:T|q^3(q^2-1)(q^3+1)\left(\frac{1}{q-1}+ \frac1{q^2-q+1} - \frac12\right) \\
&=-\frac12|S:T|q^3(q+1)\big(q^3-4q^2+2q-1\big) 
 \end{aligned}
\end{equation*}
Now, since $(q^3-4q^2-2q+1, q+1)=1$ we have a contradiction.

If $\lambda_2=2(q-1)$ then we have
\begin{equation*}
 \begin{aligned}
  \chi &= |S|(\frac1m + \frac 1n - \frac12) \\
&= |S:T|q^3(q^2-1)(q^3+1)\left(\frac{1}{2(q-1)}+ \frac1{q^2-q+1} - \frac12\right) \\
&=-\frac12|S:T|q^4(q+1)\big(q^2-3q+1\big) 
 \end{aligned}
\end{equation*}
Now, since $(q+1, q^2-3q+1)\leq 5 < q^2-3q+1$ we have a contradiction.

If $\lambda_2=4(q-1)$ then we have
\begin{equation*}
 \begin{aligned}
  \chi &= |S|(\frac1m + \frac 1n - \frac12) \\
&= |S:T|q^3(q^2-1)(q^3+1)\left(\frac{1}{2(q-1)}+ \frac1{q^2-q+1} - \frac12\right) \\
&=-\frac14|S:T|q^3(q+1)\big(2q^3-5q^2+q+1\big) 
 \end{aligned}
\end{equation*}
Now, since $(q+1, 2q^3-5q^2+q+1)$ divides $7$ we have a contradiction.

We are left with the possibility that $q=4$. Thus, writing $\Lambda=\{\lambda_1, \lambda_2\}$ we assume that $\lambda_1$ is divisible by $3$ and $\lambda_2$ is divisible by $13$. Consulting \cite{atlas} we obtain that $\lambda_2=13$ and $\lambda_1\in\{3,6,12,15\}$; we go through these one at a time:
\begin{center}
 \begin{tabular}{|c|c|}
\hline
  $\lambda_1$ & Prime dividing $\chi$ or $(S, \chi)$\\
\hline
$3$ & 7 \\
$6$ & $(T.2, -2^8\cdot 5^3)$\\
$12$ & $53$ \\
$15$ & $139$ \\
\hline
 \end{tabular}
\end{center}
The result follows.
\end{proof}

\subsection{$T=G_2(3)$}

Recall that $s=3$. Thus, writing $\Lambda=\{\lambda_1, \lambda_2\}$ we assume that $\lambda_1$ divisible by $7$ and $\lambda_2$ is divisible by $13$. We consult \cite{atlas} and find that there are two possibilities:

Suppose that $S=T=G_2(3)$; then $\Lambda=\{7,13\}$. In this case $17|\chi$ and we exclude this case. Alternatively we have $S=G_2(3).2$ and $=\{13,14\}$. In this case $\chi=-2^{12}\cdot 3^6$, a valid possibility.

\subsection{$T=Sp_6(2)$}

Recall that $s=3$. Thus, writing $\Lambda=\{\lambda_1, \lambda_2\}$ we assume that $\lambda_1$ is divisible by $5$ and $\lambda_2$ is divisible by $7$. The outer automorphism group is trivial here so $S=T$. There are three possibilities: $\Lambda=\{5,7\}$ in which case $11|\chi$ and we exclude this case; $\Lambda=\{15,7\}$ in which case $61|\chi$ and we exclude this case; $\Lambda=\{7,10\}$ in which case $\chi=-2^9\cdot 3^6$, a valid possibility.

\subsection{$T=SU_5(2)$}

Recall that $s=3$. Thus, writing $\Lambda=\{\lambda_1, \lambda_2\}$ we assume that $\lambda_1$ is divisible by $5$ and $\lambda_2$ is divisible by $11$. There are three possibilities, all of which are invalid: $\Lambda=\{5,11\}$ in which case $23|\chi$; $\Lambda=\{10,11\}$ in which case $17|\chi$; $\Lambda=\{15,11\}$ in which case $113|\chi$.

\subsection{$T=PSL_4(3)$}

Recall that $s=3$. Thus, writing $\Lambda=\{\lambda_1, \lambda_2\}$ we assume that $\lambda_1$ is divisible by $5$ and $\lambda_2$ is divisible by $13$. Consulting \cite{atlas} we see that $\lambda_1\in\{5,10,20,40\}$ and $\lambda_2\in\{13,26\}$. In all cases we find that a prime other than $2$ or $3$ divides $\chi$:
\begin{center}
 \begin{tabular}{|c|c|| c | c |}
\hline
$\Lambda$ & Prime dividing $\chi$ & $\Lambda$ & Prime dividing $\chi$ \\
\hline
\{5,13\}  & 29 & \{5,26\} & 17 \\
\{10, 13\} & 7 & \{10, 26\} & 47 \\
\{20,13\} & 97 & \{20, 26\} & 107 \\
\{40,13\} & 23 & \{40, 26\} & 227 \\
\hline
 \end{tabular}
\end{center}

\subsection{$T=PSU_4(3)$}

Recall that $s=3$. Thus, writing $\Lambda=\{\lambda_1, \lambda_2\}$ we assume that $\lambda_1$ is divisible by $5$ and $\lambda_2$ is divisible by $7$.  Consulting \cite{atlas} we see that $\lambda_1\in\{5,10,20\}$ and $\lambda_2\in \{7,14,28\}$. In all cases but two we find that a prime other than $2$ or $3$ divides $\chi$:
\begin{center}
 \begin{tabular}{|c|c|| c | c || c| c|}
\hline
$\Lambda$ & Prime dividing $\chi$ & $\Lambda$ & Prime dividing $\chi$ & $\Lambda$ & Prime dividing $\chi$ \\
\hline
\{5,7\}  & 11 & \{5,14\} & * & \{5,28\} & 37 \\
\{10, 7\} & * & \{10, 14\} & 23 & \{10, 28\} & 17 \\
\{20,7\} & 43 & \{20, 14\} & 53 & \{20, 28\} & 29 \\
\hline
 \end{tabular}
\end{center}

Since $T$ does not contain an element of order 10 nor an element of order 14, we conclude that $S=T.2$ in both cases. When $\Lambda=\{7,10\}$ we have $\chi = -2^8\cdot 3^8$; when $\lambda=\{5,14\}$ we have $\chi=-2^{11}\cdot 3^6$.

\subsection{$T=SL_4(2)$}

Recall that $s=3$. Thus, writing $\Lambda=\{\lambda_1, \lambda_2\}$ we assume that $\lambda_1$ is divisible by $5$ and $\lambda_2$ is divisible by $7$. There are three possibilities: $\Lambda=\{5,7\}$ in which case $11|\chi$ and we exclude this case; $\Lambda=\{15,7\}$ in which case $61|\chi$ and we exclude this case; $\Lambda=\{7,10\}$ in which case $\chi=-2^7\cdot 3^4$; since $t$ does not contain an element of order 10 we conclude that $S=T.2$ in this case.

\subsection{$T=SU_4(2)$}\label{s: su42}

Once again $s=3$. Since $|T|=2^6\cdot 3^4\cdot 5$ and writing $\Lambda=\{\lambda_1, \lambda_2\}$, we assume that $\lambda_1$ is divisible by $5$, while $\lambda_2$ may be any order greater than $2$. Consulting \cite{atlas} for $T$ and $T.2$ we conclude that $m\in \{5,10\}$, $n\in \{3,4,5,6,8,9,10,12\}$. If $m=5$ we exclude $n=3$ since it is well known that
$$\langle x, y \, \mid \, x^3=y^5=(xy)^2=1\rangle \cong A_5.$$
In the table below we list all possible combinations for $m$ and $n$; if $\chi$ is divisible by a prime greater than $3$ we list it, otherwise we list the value of $\chi$ as well as the isomorphism class of $S$ (either $T$ or $T.2$ or, in two cases, both):
\begin{center}
\begin{tabular}{|c|c|| c | c |}
\hline
$\Lambda$ & Prime dividing $\chi$ or $(S, \chi)$ & $\Lambda$ & Prime dividing $\chi$ or $(S, \chi)$ \\
\hline
\{5,3\} & Excluded & \{10, 3\} & $(T.2, -2^7\cdot 3^3)$ \\
\{5,4\}  & $(T, -2^4\cdot3^4)$, $(T.2, -2^5\cdot 3^4)$ & \{10,4\} & $(T.2, -2^5\cdot 3^5)$ \\
\{5,5\} & $(T, -2^5\cdot3^4)$ & \{10,5\} & $(T.2, -2^7\cdot 3^4)$ \\
\{5,6\} & $(T, -2^7\cdot3^3)$, $(T.2, -2^8\cdot 3^3)$ & \{10,4\} & 7 \\
\{5,8\} & 7 & \{10,8\} & 11 \\
\{5,9\} & 17 & \{10,9\} & 13 \\
\{5,10\} & Already covered & \{10, 10\} & $(T.2, -2^6\cdot 3^5)$ \\
\{5,12\} & 13 & \{10,12\} & 19 \\
\hline
\end{tabular}
\end{center}

\subsection{Alternating groups}

We must consider $S=A_7, S_7, A_9, S_9$; in all cases $s=3$ thus, writing $\Lambda=\{\lambda_1, \lambda_2\}$ we assume that $\lambda_1$ is divisible by $5$ and $\lambda_2$ is divisible by $7$. Consulting \cite{atlas} we see that $\lambda_1\in\{5,10,15,20\}$ and $\lambda_2\in\{7,14\}$. In all but two cases we see that a prime greater than $3$ divides $\chi$:

\begin{center}
 \begin{tabular}{|c|c|| c | c |}
\hline
$\Lambda$ & Prime dividing $\chi$ & $\Lambda$ & Prime dividing $\chi$ \\
\hline
\{5,7\}  & 11 & \{5,14\} & * \\
\{10, 7\} & * & \{10, 14\} & 23  \\
\{15,7\} & 61 & \{15,14\} & 19 \\
\{20,7\} & 43 & \{20, 14\} & 53  \\
\hline
 \end{tabular}
\end{center}

Thus we must check $\Lambda=\{10,7\}$ and $\{5,14\}$ for the four different groups. When $S=A_7$ neither of these are possible. When $S=S_7$ only $\Lambda=\{10,7\}$ is possible and we obtain $\chi=-2^4\cdot3^4$. When $S=A_9$ only $\Lambda=\{10,7\}$ is possible and we obtain $\chi=-2^6\cdot3^6$. Finally when $S=S_9$ both cases are possible and we obtain $\chi=-2^9\cdot3^4$ when $\Lambda=\{5,14\}$ and $\chi=-2^7\cdot3^6)$ when $\Lambda=\{10,7\}$.

\subsection{Sporadic groups}\label{s: sporadics}

We must consider $S=M_{11}, M_{12}, M_{12}.2$. Since $s=3$ the elements of $\Lambda$ must be divisible by 5 and 11. Examining \cite{atlas} we see that there are only two possibilities in total: $\Lambda=\{5,11\}$ (in which case 23 divides $\chi$) or $\Lambda=\{10,11\}$ (in which case 17 divides $\chi$).

\subsection{Existence}\label{s: existence}

The work of Sections \ref{s: 61} to \ref{s: sporadics} has yielded a number of putative $(2,m,n)$-groups for which we must now establish existence or otherwise. When $T=PSL_2(q)$ for some $q\geq 4$ we have the following possibilities for $S$ provided $S\neq PSL_2(q)$ or $PGL_2(q)$:

\begin{center}
 \begin{tabular}{|c|c|c|}
 \hline
$S$ & $\{m,n\}$ & $\chi$ \\
\hline
$PSL_2(9).2$ & $\{4,5\} $ & $-2^2\cdot 3^2$ \\
$PSL_2(9).2$ & $\{5,6\} $ & $-2^5\cdot 3$ \\
$PSL_2(9).(C_2\times C_2)$ & $\{4,10\} $ & $-2^3 \cdot 3^3$ \\
$SL_2(16).2$ & $\{6,5\}$ & $-2^6\cdot 17$ \\
$SL_2(16).2$, & $\{10,3\}$ & $-2^5\cdot 17$ \\
$PSL_2(25).2$ & $\{6,13\}$ & $-2^5\cdot 5^3$ \\
\hline
 \end{tabular}
\end{center}

In all cases bar the first the requirement that $S$ be generated by a pair of elements of order $m$ and $n$ uniquely prescribes the group up to isomorphism.

Let us consider the exceptional first case. Then $S=PSL_2(9).2$, $\{m,n\}=\{4,5\}$ and there are two isomorphism classes for $S$ that we must consider, namely $S=M_{10}$ and $S=S_6$. Suppose that $S=M_{10}$ and let $S=\langle g, h \rangle$ where $o(g)=4$ and $o(h)=5$. Then $g\not\in PSL_2(9)$ and $h\in PSL_2(9)$. Thus $gh\not\in PSL_2(9)$. But $M_{10}$ is a non-split extension and so $o(gh)\neq 2$ which is a contradiction.

On the other hand suppose that $S=S_6$; then \cite{conder2} implies that $S$ is not a $(2,4,5)$-group and this case is also excluded. On the other hand \cite{conder2} implies that $S$ is a $(2,5,6)$-group which confirms the existence of the group in the second line of the table.

For the next three lines we use a combination of \cite{gap} and \cite{magma}; these rule out both possibilities when $S=SL_2(16).2$. On the other hand they confirm that $PSL_2(9).(C_2\times C_2)$ is a $(2,4,10)$-group.

We are left with the case when $S=PSL_2(25).2$. Note first that the list of maximal subgroups of $PSL_2(25)$ given in \cite{atlas} implies that any pair of elements of order $3$ and $13$ in $PSL_2(25)$ must generate $PSL_2(15)$. Thus it is enough to show that there are elements $g,h\in S\backslash PSL_2(25)$ such that $o(g)=2, o(h)=6$ and $gh\in PSL_3(4)$ is of order $13$; a simple application of Proposition~\ref{p: character} confirms that such elements exist. We have justified the entries in Table \ref{table: main1}. 

Now we turn to the situation where $T\neq PSL_2(q)$ for any $q\geq 4$. Table \ref{table: main} lists twenty-seven pairs $(S, \{m,n\})$ such that $S$ is a $(2,m,n)$-group. (The total of twenty-seven takes into account two key facts: when only $T$ is specified, there are two groups to consider for $S$; when $(S,\{m,n\})= (PSU_4(3).2, \{10, 7\})$, we must consider three isomorphism classes for $S$.)

Our work in Sections \ref{s: 61} to \ref{s: sporadics} implies that there are a number of other possible pairs to consider. We list them as follows:

\begin{center}
 \begin{tabular}{|c|c|c|}
 \hline
Group & $\{m,n\}$ & $\chi$ \\
\hline
  $S=SU_3(3)$& $\{3,7\}$& $-2^4\cdot 3^2$ \\
  $S=SU_3(3)$& $\{4,7\}$& $2^3\cdot 3^4$ \\
$T=SU_4(2)$ & $\{5,4\}$ & $-|S:T|\cdot2^4\cdot 3^4$ \\
$S=SU_4(2)$ & $\{5,5\}$ & $-2^5\cdot 3^4$ \\
$S=SU_4(2).2$ & $\{10,3\}$ & $-2^7\cdot 3^3$ \\
 $S=S_9$& $\{10,7\}$& $-2^7\cdot3^6$ \\
 $S=S_9$& $\{5,14\}$& $-2^{10}\cdot3^4$ \\
\hline 
 \end{tabular}
\end{center}

We follow the conventions of Table \ref{table: main}; in particular when we specify only $T$ we must consider two groups $S=T$ and $S=T.2$. Our first job is to rule out the eight possibilities listed in this table. The first possibility is excluded by \cite{conder3} in which it is shown that $SU_3(3)$ is not a Hurwitz group. 

Now consider the second possibility when $S=SU_3(3)$ and $\{m,n\}=\{4,7\}$. We consult \cite{atlas} to find that $SU_3(3)$ has a unique conjugacy class of involutions and three conjugacy classes of elements of order $4$ which we label, as per \cite{atlas}, 4A, 4B and 4C. Let $g$ be an involution and $h$ an element of order $4$; Proposition~\ref{p: character} implies that if $h$ is in conjugacy class 4B or 4C then $gh$ is never of order $7$, so suppose that $h$ is in conjugacy class 4A. Then Proposition~\ref{p: character} implies that, for any $z\in S$ of order $7$ there are seven pairs $(x,y)\in g^S\times h^S$ such that $xy=z$. Now an application of Proposition~\ref{p: character} to $H=PSL_2(7)$ implies that, for any $z\in S$ of order $7$ there are seven pairs $(x,y)\in H$ such that $o(x)=2$, $o(y)=4$ and $xy=z$. Furthermore \cite{atlas} implies that $S$ has a subgroup isomorphic to $PSL_2(7)$. Since a Sylow $7$-subgroup of $H$ is cyclic of order $7$, every element of order $7$ in $S$ lies in a subgroup of $S$ isomorphic to $H$ and we conclude that any pair $(x,y)\in S$ such that $o(x)=2, o(y)=4$ and $o(xy)=7$ must lie in a subgroup isomorphic to $PSL_2(7)$ and so cannot generate $S$.

The four almost simple groups $S$ with socle $T\cong SU_4(2)$ can all be ruled out using \cite{gap} or \cite{magma}. The same is true of the final two cases involving $S_9$, although we give an alternative proof using the following result \cite{conmc}.

\begin{prop}\label{p: trans}
Suppose that $G\leq S_n$ and $G$ is generated by elements $g_1,g_2, ...,g_s$ where $g_1\cdots g_s=1$.  Suppose that, for $i=1,\dots, s$, the generator $g_i$ has exactly $c_i$ cycles on $\Omega=\{1,\dots, n\}$ and that $G$ is transitive on $\Omega$, then
$$\sum\limits_{i=1}^s c_i +2 \leq n(s-2).$$
\end{prop}

We apply this to $G=S_9$ with $s=3$; observe that if $z\in G$ is an involution, then $z$ has at least $5$ cycles. If $g$ is of order $10$ then it has at least $3$ cycles and if $h$ has order $7$ then it has at least $3$ cycles; since $5+3+3>9$ we conclude that $G$ is not a $(2,7,10)$-group. Similarly if $g$ is of order $5$ then it has at least $5$-cycles; since $5+5>9$ we conclude that $G$ is not a $(2,5,k)$-group for any $k$.

All that remains is to show that the twenty-seven pairs listed in Table \ref{table: main} correspond to a $(2,m,n)$-group. In nearly all cases we can confirm this easily using \cite{gap} or \cite{magma}; we mention three cases that are slightly tricky and which we prefer to do ``by hand''.

Consider first the two cases 
$$(S,\{m,n\})=(PSL_3(4).2_2, \{5,14\})\textrm{ and }(S,\{m,n\})=(PSL_3(4).2_3, \{7,10\}).$$ (We use \cite{atlas} notation to single out the particular degree $2$ extension to be studied in each case.) Note first that the list of maximal subgroups of $PSL_3(4)$ given in \cite{atlas} implies that any pair of elements of order $5$ and $7$ in $PSL_3(4)$ must generate $PSL_3(4)$. Thus in the first instance it is enough to show that there are elements $g,h\in S\backslash PSL_3(4)$ such that $o(g)=2, o(h)=14$ and $gh\in PSL_3(4)$ is of order $5$; a simple application of Proposition~\ref{p: character} confirms that such elements exist. Similarly in the second instance it is enough to show that there are elements $g,h\in S\backslash PSL_3(4)$ such that $o(g)=2, o(h)=10$ and $gh\in PSL_3(4)$ is of order $7$; again Proposition~\ref{p: character} confirms that such elements exist.

Finally suppose that $(S,\{m,n\})=(G_2(3).2, \{13,14\})$. Let $z$ be an element of order $13$ in $G_2(3)$. Proposition~\ref{p: character} implies that the number of pairs of elements $g,h\in S$ such that $o(g)=2$ and $o(h)=14$ is 286.

Now \cite{atlas} implies that the only maximal subgroup of $G_2(3).2$ that contains elements of order $13$ and of order $14$ is $PSL_2(13):2=PGL_2(13)$. Let $M$ be a maximal subgroup isomorphic to $PGL_2(13)$ that contains $z$. The number of pairs of elements $g,h\in M$ such that $o(g)=2$ and $o(h)=14$ is 13.

Let $P$ be a Sylow $13$-subgroup of $S$ lying in $M$. Then $N_S(P)< M$ and we conclude that every Sylow 13-subgroup lies in a unique maximal subgroup isomorphic to $PGL_2(13)$.
Thus there are $286 - 13 = 273$ pairs of elements $g,h\in S$ such that $o(g)=2$, $o(h)=14$, $gh=z$ and $\langle g, h\rangle \not\leq M$. Thus $\langle g,h\rangle$ does not lie in any maximal subgroup and we conclude that $\langle g,h\rangle=S$ as required.

\section{Closing remarks}\label{s: final}

There are a number of obvious avenues for future research; we briefly run through some of them.

\subsection{Improving Theorem \ref{t: almost simple two primes}}

The obvious weakness with Theorem \ref{t: almost simple two primes} is that those $(2,m,n)$-groups $(S,g,h)$ for which $S=PSL_2(q)$ or $S=PGL_2(q)$, for some $q$, are not classified. Indeed we have not even been able to establish whether or not there are an infinite number of such $(2,m,n)$-groups with $\chi=-2^as^b$.

The nature of the problem is illustrated by the following example: suppose that $S=PSL_2(2^x)$ for some integer $x>1$. Set $m=2^x+1$ and $n=2^x-1$; using a knowledge of the subgroups of $S$, the character table of $S$, and Proposition \ref{p: character} one can quickly deduce that $S$ is a $(2,m,n)$-group. Writing $q=2^x$, the Euler characteristic $\chi$ is equal to $-\frac12q(q^2-4q+1)$. Thus, if $\chi$ is to be divisible by exactly two distinct primes, then we must have 
\begin{equation}\label{number theory}
q^2-4q+1=s^b
\end{equation}
for some odd prime $s$ and positive integer $b$. The number theoretic task of describing those $q,s$ and $b$ such that (\ref{number theory}) holds true appears to be difficult.

\subsection{Three primes}

Consider those $(2,m,n)$-groups $G$ with associated Euler characteristic divisible by exactly three distinct primes. In this case the analogue of Proposition~\ref{p: two primes} is slightly more complicated as the group $G$ may have non-simple non-abelian chief factors. 

In particular, if $\chi=-2^as^bt^c$ then the group $G$ may have a chief factor isomorphic to $T^k$ for some simple group $T$ and $k>1$; in this case $|T|$ must be divisible by exactly three primes (namely $2, s$ and $t$) and such groups do exist. (There are precisely eight simple groups whose orders are divisible by exactly three primes, namely $A_5$, $A_6$, $\text{PSp}_4(3)$, $PSL_2(7)$, $PSL_2(8)$, $PSU_3(3)$, $PSL_3(3)$ and $PSL_2(17)$; this fact is not dependent on the classification of finite simple groups; see, for example, \cite{k4}.) 

\subsection{Other possibilities}

Each entry of Tables \ref{table: main1} and \ref{table: main} warrants further investigation. Although we know that each entry corresponds to at least one $(2,m,n)$-group we have not established how many distinct $(2,m,n)$-groups occur in each case. Once we have established this fact, there are a plethora of further questions: for instance, which of them are {\it reflexible} (i.e. admit an orientation-reversing automorphism)?

In a different direction much of the work of this paper will carry over to the study of groups associated with {\it non-orientable} regular maps; indeed in this situation the structure of the group has more properties that we can exploit (for instance it is generated by three involutions) and we intend to address this question in a future paper.

\bibliographystyle{amsalpha}
\bibliography{paper}

\end{document}